\documentclass[final]{siamart}
\usepackage[margin=1in]{geometry}
\usepackage{color}
\usepackage{amsbsy}
\usepackage{amstext}
\usepackage{amssymb}
\usepackage{stmaryrd}
\usepackage{graphicx}
\usepackage[caption = false]{subfig}
\usepackage{float}
\usepackage{comment}
\usepackage[numbers,sort]{natbib}
\newsiamremark{rem}{Remark}
  
\providecommand{\tabularnewline}{\\}

\makeatletter
\def\blfootnote{\gdef\@thefnmark{}\@footnotetext}
\makeatother

\newcommand{\cblue}[1]{{\color{black}{#1}}}
\newcommand{\cred}[1]{{\color{black}{#1}}}
\newcommand{\cgreen}[1]{{\color{black}{#1}}}

\newcommand{\TheTitle}{DSA Preconditioning for DG discretizations of $S_{N}$\\transport on High-Order curved meshes}
\newcommand{\TheAuthors}{T. S. Haut, B. S. Southworth, P. G. Maginot, and V. Z. Tomov}
\headers{DSA Preconditioning for DG discretizations of $S_{N}$ transport}{\TheAuthors}
\title{{\TheTitle}}

\author{  Terry~S.~Haut
  \thanks{Lawrence Livermore National Laboratory,
          (\email{haut3@llnl.gov}, \email{tomov2@llnl.gov}).}
\and
  Ben~S.~Southworth
  \thanks{Department of Applied Mathematics,
          University of Colorado at Boulder
          (\email{ben.s.southworth@gmail.com}).}
\and
  Peter~G.~Maginot
  \thanks{Los Alamos National Laboratory
          (\email{pmaginot@lanl.gov}).}
\and
  Vladimir Z. Tomov\footnotemark[2] 
}

\ifpdf\hypersetup{  pdftitle={\TheTitle},
  pdfauthor={\TheAuthors}
}
\fi

\begin{document}
\maketitle
\allowdisplaybreaks 

\begin{abstract}
This paper derives and analyzes new diffusion synthetic acceleration (DSA) preconditioners 
for the $S_N$ transport equation when discretized with a high-order (HO) discontinuous Galerkin (DG) discretization. DSA preconditioners
address the need to accelerate the $S_N$ transport equation when the mean free path $\varepsilon$ of particles is small and the condition number of
the $S_N$ transport equation scales like  $\mathcal{O}\left( \varepsilon^{-2} \right)$.
By expanding the $S_N$ transport operator in  $\varepsilon$ and employing a rigorous singular matrix perturbation
analysis, we derive a DSA matrix that reduces to the symmetric interior penalty (SIP) DG discretization of the standard
continuum diffusion equation when the mesh is first-order and the total opacity is constant. We prove that preconditioning the 
HO DG $S_N$ transport equation with the SIP DSA matrix results in an $\mathcal{O}\left( \varepsilon \right)$ perturbation of
the identity, and fixed-point iteration therefore converges rapidly for optically thick problems. 
However, the SIP DSA matrix is conditioned like $\mathcal{O}\left( \varepsilon^{-1} \right)$, making it difficult to invert for small $\varepsilon$.
We further derive a new two-part, additive DSA preconditioner based on a continuous Galerkin discretization of diffusion-reaction, which
has a condition number independent of $\varepsilon$, and prove that this DSA variant has the same theoretical efficiency as the
SIP DSA preconditioner in the optically thick limit.
The analysis is extended to the case of HO (curved) meshes, where so-called mesh cycles can result from elements both 
being upwind of each other (for a given discrete photon direction). In particular, we prove that performing two additional transport 
sweeps, with fixed scalar flux, in between DSA steps yields the same theoretical conditioning of fixed-point iterations as in the cycle-free case.
Theoretical results are validated by numerical experiments on a HO, highly
curved 2D and 3D meshes that are generated from an arbitrary Lagrangian-Eulerian hydrodynamics code, where the additional inner
sweeps between DSA steps offer up to a $4\times$ reduction in total number of sweeps required for convergence.
\end{abstract}
  
\blfootnote{This work was performed under the auspices of the U.S. Department of Energy by Lawrence Livermore National Laboratory under contracts DE-AC52-07NA27344, B614452, and B627942, Lawrence Livermore National Security, LLC (LLNL-JRNL-759881). This work was  performed under the auspices of the U.S. Department of Energy under grant number (NNSA) DE-NA0002376. Disclaimer: This document was prepared as an account of work sponsored by an agency of the United States government. Neither the United States government nor Lawrence Livermore National Security, LLC, nor any of their employees makes any warranty, expressed or implied, or assumes any legal liability or responsibility for the accuracy, completeness, or usefulness of any information, apparatus, product, or process disclosed, or represents that its use would not infringe privately owned rights. Reference herein to any specific commercial product, process, or service by trade name, trademark, manufacturer, or otherwise does not necessarily constitute or imply its endorsement, recommendation, or favoring by the United States government or Lawrence Livermore National Security, LLC. The views and opinions of authors expressed herein do not necessarily state or reflect those of the United States government or Lawrence Livermore National Security, LLC, and shall not be used for advertising or product endorsement purposes. }

\section{Introduction}
\subsection{Background}

The $S_{N}$ transport equation forms a key component in modeling
the interaction of radiation and a background medium, and its accurate
solution is critical in the simulation of astrophysics,
interial confinement fusion, and a number of other fields. In this paper,
we derive and analyze diffusion-based preconditioners for a high-order
(HO) discontinuous Galerkin (DG) discretization of the monoenergetic
$S_{N}$ transport equations in the challenging (but typical) case
of scattering-dominated regimes. One motivation of this research is
in the context of HO arbitrary Lagrangian-Eulerian (ALE) hydrodynamics
on HO (curved) meshes \cite{blastALE}\textcolor{blue}, where standard diffusion-based preconditioners
are inadequate.

The standard approach for solving the $S_{N}$ transport equations
involves a fixed-point iteration, referred to as source iteration
in the transport literature. It is well known that source iteration
can converge arbitrarily slowly in the optically thick limit of large
scattering and small absorption. To quantify this, it is useful to introduce the diffusion scaling. In particular, let $\varepsilon$
be a non-dimensional parameter representing the ratio of a typical mean free path of a particle to a dimension of the domain
\cite{larsen1974asymptotic}. \cred{ Then, for characteristic mesh spacing $h_{\mathbf{x}}$ and total cross section $\sigma_t$, the optically
thick limit corresponds to $ \varepsilon/\left( h_{\mathbf{x}} \sigma_t \right) \sim \varepsilon^2  \sigma_a/ \sigma_t  \ll 1$ }
In this case, the matrix corresponding to source iteration
has a condition number that scales like $ \left( h_{\mathbf{x}} \sigma_t /\varepsilon \right) ^{2}$ and,
therefore, will converge very slowly without specialized preconditioners.
Such preconditioners typically involve a two-level acceleration scheme
and fall within two broad classes: (i) using a diffusion equation
to solve for a corrected scalar flux, referred to as diffusion synthetic
acceleration (DSA), and (ii) solving the $S_{N}$ transport equations
with a reduced number of angular quadrature points, referred to as
transport synthetic acceleration (TSA) \cgreen{(cf. \cite{Ramone-Adams-Nowak-1997}, \cite{Zika-Adams-2000}, \cite{Larsen-Nowak-Hanshaw-2003}). This paper focuses on DSA-type
algorithms. An excellent discussion on the historical development
of DSA can be found in \cite{Larsen_1984}.}

Some of the earliest work on accelerating transport equations with
a diffusion-based preconditioner can be found in, for example,
\cite{Lebedev_1969,Marchuk_Lebedev_1986,Kopp_1963}. It was shown
in \cite{Gelbard_Hageman_1969,Reed_1971} that diffusion-based
acceleration for source iteration is effective for fine spatial meshes ($ \varepsilon \geq h_{\mathbf{x}} \sigma_t $)
but its performance can degrade for coarse meshes (that is, $ \varepsilon \ll h_{\mathbf{x}} \sigma_t $ ). Further seminal
work in \cite{Alcouffe-1977} contained a derivation and theory for
a diffusion-equation accelerator whose discretization is consistent
with the $S_{N}$ transport diamond-difference scheme (a finite-volume
type scheme for transport), and which yields fast acceleration independent
of the spatial mesh size. Since, DSA methods have been significantly refined
and expanded to other spatial discretizations \cite{Larsen-1982,
Larsen-McCoy-1982,Larsen_1984,Adams-William-1992,Wareing-2001,Warsa-Wareing-Morel-2002,Adams-Larsen-2002}.
In the context of HO DG discretizations, the authors in \cite{Wang-Ragusa-2010} 
develop a modified symmetric interior penalty (MIP) DSA scheme
for HO DG (on first-order meshes) and numerically demonstrate that source iteration
converges rapidly with the MIP DSA preconditioner. 

\cblue{
Here we present a rigorous, discrete analysis of DSA in the context of HO DG,
on potentially HO (curved) meshes. The paper proceeds as follows. Section \ref{sec:High-order-DG-and}
introduces the DG discretization of the $S_{N}$ transport equations, as well as
the standard fixed-point iteration to solve the discrete $S_{N}$ system, known as ``source iteration.''
The primary theoretical contributions are formally stated in
Section \ref{subsec:Statement-of-theorems}, with the proofs provided in Section
\ref{sec:Proofs}. As a byproduct of this analysis, three new DG DSA preconditioners are developed that are effective for HO discretizations on HO meshes. 
Two of these preconditioners reduce to interior penalty DG discretizations of diffusion when the mesh is straight-edged (i.e. non-curved); the third preconditioner
is new even for straight-edged meshes, and avoids the numerical difficulties associated with inverting the interior penalty DSA matrices while provably having the same efficacy in the thick limit. Section
\ref{subsec:Connection to previous work} relates our analysis to both the modified
interior penalty (MIP) preconditioner \cite{Wang-Ragusa-2010}
and the consistent DSA preconditioner \cite{Warsa-Wareing-Morel-2002}. 
In Section \ref{sec:numerical}, the efficacy of our DSA preconditioners is demonstrated for
HO DG discretizations on highly curved 2D and 3D meshes generated by \cite{blastALE}
(a HO ALE hydrodynamics code). With the newly developed HO DSA algorithm and DSA
discretization, rapid fixed-point convergence is obtained for all tested values of
the mean free path (while iterations diverge on the HO mesh without the proposed
algorithmic modification). Numerical results also demonstrate the new additive DSA
preconditioner to be robust on optically thick and thin problems in one spatial dimension,
with preconditioning in the optically thick limit equally as effective as traditional
DSA using our new discretization. Brief conclusions are given in Section \ref{sec:Conclusions}. 

\subsection{Outline of contributions}

In general, the discrete source-iteration propagation operator has singular
modes with singular values on the order of $\mathcal{O}(\varepsilon^{2})$, where $\varepsilon$ is
the characteristic mean free path. These modes are referred to as the near nullspace of source iteration
and are extremely slow to converge when $\varepsilon\ll 1$. By directly expanding the discrete
DG source-iteration operator in $\varepsilon$, we derive a DSA preconditioner that exactly
represents the problematic error modes that are slow to decay. 
For first-order meshes and constant opacities, we also show that the DSA matrix exactly
corresponds to the symmetric interior penalty (SIP) DG discretization of the diffusion equation.
In Theorem~\ref{th:MIP}, we prove that the corresponding DSA-preconditioned $S_{N}$ transport equations
is an $\mathcal{O}\left(\varepsilon\right)$ perturbation of the identity,
and the resulting fixed-point iteration therefore converges rapidly for sufficiently small mean
free path. In the optically thick limit of $ \varepsilon/\left( h_{\mathbf{x}} \sigma_t \right) \sim \varepsilon^2  \sigma_a/ \sigma_t  \ll 1$
(and assuming constant total opacity and a first order mesh), this diffusion discretization
is identical to the MIP DSA preconditioner that is numerically analyzed in \cite{Wang-Ragusa-2010}, and  Theorem~\ref{th:MIP} provides a rigorous justification for its efficacy.
In Section \ref{subsec:Connection to previous work} we discuss
stabilization in thin regimes, and formulate a nonsymmetric interior penalty
(IP) DSA preconditioner as an alternative to the SIP DSA approach.

It turns out the SIP DSA matrix is in the form of a
singular matrix perturbation: the dominant term is of order $1/\varepsilon$
relative to the other terms, and has a nullspace consisting of continuous
functions with zero boundary values. This term acts as a large penalization and
constrains the solution to be continuous in the limit of $\varepsilon\rightarrow0$.
This term also leads to the SIP DSA matrix having a condition number
that scales like $\mathcal{O}(1/\varepsilon)$, one of the primary reasons that
DG DSA discretizations such as SIP can be difficult to precondition (although see \cite{O'Malley_Kophazi_2017, Warsa-Wareing-Morel-2003,  Antonietti-Sarti-Verani-Zikatanov-2017} for 
several approaches to preconditioning these systems). 
Appealing to the singular perturbation, we then derive a two-part additive DSA
preconditioner based on projecting onto the spaces of continuous and discontinuous
functions. Theorem~\ref{th:new_DSA} proves that the resulting preconditioned
$S_{N}$ transport equation fixed-point iteration is also an $\mathcal{O}\left(\varepsilon\right)$
perturbation of the identity and therefore has the same theoretical efficiency as
the SIP DSA preconditioner in the optically thick limit. Moreover, the condition numbers
of linear systems in the additive preconditioner are independent of $\varepsilon$.
We note that the leading order term in this two-part additive DSA preconditioner
corresponds to the continuous Galerkin (CG) discretization of the
diffusion equation obtained in \cite{Guermond-Kanschat-2010}.

Next, we modify the analysis to account for HO curved meshes. In
this case, neighboring mesh elements can both be upwind of each other,
leading to so-called mesh cycles. With mesh cycles, the discrete
streaming plus collision operator that is inverted in source iteration
is no longer block lower triangular in any element ordering, and so it 
cannot be easily inverted through a forward solve.
We prove in Theorem~\ref{th:mesh_cycles} that performing two additional
transport sweeps on the $S_{N}$ transport equations, with a fixed
scalar flux, yields a preconditioner that has the same asymptotic
efficiency as that obtained on cycle-free meshes.

Finally, we perform a series of numerical tests to compare the behavior,
in thick and thin regimes, of the three major preconditioning approaches
presented in this work, namely, the SIP DSA (Theorem \ref{th:MIP}),
its IP modification (Section \ref{sec:consistent discretization}), and the
additive DSA preconditioner (Theorem \ref{th:new_DSA}).}

\section{High-Order (HO) Discontinuous Galerkin (DG) discretization of $S_N$ transport and the need for preconditioning in scattering dominated regimes}\label{sec:High-order-DG-and}

\subsection{DG discretization}

Consider the mono-energetic, steady-state, discrete-ordinates linear
Boltzmann equation \cred{with isotropic scattering}, given by
\begin{equation}
\begin{split}
\boldsymbol{\Omega}_{d}\cdot\nabla_{\mathbf{x}}\psi_{d}\left(\mathbf{x}\right)+\frac{\sigma_{t}\left(\mathbf{x}\right)}{\varepsilon}\psi_{d}\left(\mathbf{x}\right) & =\frac{1}{4\pi}\left(\frac{\sigma_{t}\left(\mathbf{x}\right)}{\varepsilon}-\varepsilon\sigma_{a}\left(\mathbf{x}\right)\right)\sum_{d'=1}^{N_{\Omega}}w_{d'}\psi_{d'}\left(\mathbf{x}\right)+\varepsilon q_{d}\left(\mathbf{x}\right),  \,\,\,\, \mathbf{x} \in \mathcal{D} \\
\psi_{d}\left(\mathbf{x}\right) & =\psi_{d,\text{inc}}\left(\mathbf{x}\right),\,\,\,\,\mathbf{x}\in \partial{\mathcal{D}} \,\,\,\,\text{and}\,\,\,\,\mathbf{n}\left(\mathbf{x}\right)\cdot\boldsymbol{\Omega}_{d}<0.
\label{eq:linear SN transport}
\end{split}
\end{equation}
\cred{In equation (\ref{eq:linear SN transport}), $\mathcal{D}$ denotes the spatial domain with boundary $\partial{\mathcal{D}}$, $\psi_{d}\left(\mathbf{x}\right)$ denotes the specific intensity associated with the discrete ordinate direction $\boldsymbol{\Omega}_{d}$, and  $q_{d}\left(\mathbf{x}\right)$ denotes a fixed (direction-dependent) source.}
Here, the total opacity, $\varepsilon\sigma_{t}^{-1}\left(\mathbf{x}\right)$,
and the absorption opacity, $\varepsilon\sigma_{a}\left(\mathbf{x}\right)$,
are scaled according to the diffusion limit, where $\varepsilon$
is a non-dimensional parameter proportional to the characteristic mean free path and which goes to zero in the optically
thick limit \cite{larsen1974asymptotic}. The quadrature angle vectors $\boldsymbol{\Omega}_{d}\in\mathbb{S}^{2}$
and weights $w_{d}>0$ are constructed to have desirable symmetry
properties and integrate spherical harmonics up to a given degree
that depends on the number of angles, $N_{\Omega}$.

We consider a discontinuous Galerkin (DG) discretization of the $S_{N}$
transport equation. To do so, we set some notation.
First, consider a decomposition of the domain $\mathcal{D}$ in to a set $\mathcal{E}$ of
\cgreen{high-order (curved)} elements $ \kappa \in \mathcal{E}$,
and let $\mathcal{F}$ denote the set of interior and boundary finite
element faces $\Gamma \in \mathcal{F}$. \cred{We further decompose the set $\mathcal{F} = \mathcal{F}_{\text{int}} \cup \mathcal{F}_{\text{ext}}$
into the set of interior$ \mathcal{F}_{\text{int}}$ faces and the set  $\mathcal{F}_{\text{ext}}$ of boundary faces.}
\cgreen{
The finite element space $\mathcal{U}$ corresponds to the collection of
piecewise polynomial functions of fixed degree $r$ on each reference element
$\hat{\kappa}$, $\mathcal{P}_r\left( \hat{\kappa} \right)$, 
$$
 \mathcal{U} = \left\{ u \in L^2 \left( \mathcal{D} \right) :
                       \hat{u} \in \mathcal{P}_r\left( \hat{\kappa} \right)
               \right\} .
$$
The values of $u$ \cred{in} physical space are obtained simply by
$u(x) = \hat{u}(\hat{x}) = \sum_{i = 1}^{N_{\kappa}} \hat{v}_i(\hat{x}) u_i$,
where $\hat{x} \rightarrow x$ is the mapping from reference to physical coordinates,
$\{ \hat{v}_i \}_1^{N_{\kappa}}$ is the basis of $\mathcal{U}$ on $\hat{\kappa}$, and
$\boldsymbol{u} = (u_1 \dots u_{N_{\kappa}})$ are the finite element coefficients of u.
Basis functions in physical space are also obtained by $v(x) = \hat{v}(\hat{x})$.
The order $r$ of the solution space $\mathcal{U}$ is generally
independent of the order of the mesh.
All methods presented in this work are fully algebraic and do not
involve geometric operations, thus they are independent of the discrete mesh
representation; interested readers can find technical details about our mesh
representation approach in \cite{Dobrev2012}.
}
For an interior mesh face $\Gamma \in \mathcal{F}$ shared by two neighboring elements $\kappa$ and $\kappa'$, we let $\mathbf{n}$ denote the normal vector that points from $\kappa$ and $\kappa'$. Given this (fixed but arbitrary) choice for the sign of the normal vector $\mathbf{n}$ on each element face, the jump $\left\llbracket u\right\rrbracket $
and average $\left\{ u\right\} $ for a function $u \in \mathcal{U}$ are defined by
\[
\left\llbracket u\right\rrbracket =\begin{cases}
u_{\kappa}-u_{\kappa'}, & \,\,\,\,\,\,\text{if \ensuremath{\Gamma} is an interior face shared by elements \ensuremath{\kappa} and \ensuremath{\kappa'}},\\
u_{\kappa}, & \,\,\,\,\,\,\text{\text{if \ensuremath{\Gamma} is a boundary face of element \ensuremath{\kappa},}}
\end{cases}
\]
and 
\[
\left\{ u\right\} =\begin{cases}
\left(u_{\kappa}+u_{\kappa'}\right)/2, & \,\,\,\,\,\,\text{if \ensuremath{\Gamma} is an interior face shared by elements \ensuremath{\kappa} and \ensuremath{\kappa'}},\\
u_{\kappa}, & \,\,\,\,\,\,\text{\text{if \ensuremath{\Gamma} is a boundary face of element \ensuremath{\kappa}}.}
\end{cases}
\]
Although the definitions of the jump $\left\llbracket u\right\rrbracket$ and average $\left\{ u\right\}$ depend on arbitrarily choosing a sign for the normal vector $\mathbf{n}$, it turns out that the bilinear forms below are invariant with respect to this choice.

Following the standard DG discretization procedure and using
upwinding to define the numerical flux, (\ref{eq:linear SN transport})
can be discretized as
\begin{equation}
\boldsymbol{\Omega}_{d}\cdot\mathbf{G}\boldsymbol{\psi}^{(d)}+F^{(d)}\boldsymbol{\psi}^{(d)}+\frac{1}{\varepsilon}M_{t}\boldsymbol{\psi}^{(d)}-\frac{1}{4\pi}\left(\frac{1}{\varepsilon}M_{t}-\varepsilon M_{a}\right)\boldsymbol{\varphi}=\frac{1}{4\pi}\left(\boldsymbol{q}_{\text{inc}}^{(d)}+\varepsilon\boldsymbol{q}^{(d)}\right).\label{eq:Tmat1}
\end{equation}
Here the vector $\boldsymbol{\varphi}$ of coefficients for the scalar flux $\varphi$ is given by
\begin{equation}
\boldsymbol{\varphi}=\sum_{d}w_{d}\boldsymbol{\psi}^{(d)},\label{eq:scalar flux}
\end{equation}
the vectors $\boldsymbol{q}_{\text{inc}}^{(d)}$ and $ \boldsymbol{q}^{(d)}$ on
the right hand side of (\ref{eq:Tmat1})  correspond to the linear forms
\begin{align}
\left[ \boldsymbol{q}_{\text{inc}}^{(d)} \right]_m & = -
  \sum_{\Gamma\in\mathcal{F_{\text{ext}}}}
  \int_{\Gamma} \boldsymbol{\Omega}_{d} \cdot \mathbf{n}
                v_m \psi_{\text{inc}}^{(d)}~dS +
  \frac{1}{2}\sum_{\Gamma\in\mathcal{F_{\text{ext}}}}
  \int_{\Gamma} \left|\boldsymbol{\Omega}_{d}\cdot\mathbf{n} \right|
                v_m\psi_{\text{inc}}^{(d)}~dS, \\
\left[ \boldsymbol{q}^{(d)} \right]_m & =
  \sum_{\kappa \in \mathcal{E}}\int_{\kappa} v_m q^{(d)} d\mathbf{x},
\end{align}
where $\{ v_m \}_1^N$ is the finite element basis of $\mathcal{U}$, and $N$ is
the total number of degrees of freedom in $\mathcal{U}$.
We will also denote by $\mathbf{u}$ and $\mathbf{v}$ the vectors of coefficients
corresponding to some discrete functions $u$ and $v$ in the
finite element space $\mathcal{U}$.
The matrices $\boldsymbol{\Omega}_{d}\cdot\mathbf{G}$, $F^{(d)}$,
$M_{t}$, and $M_{a}$ in equation (\ref{eq:Tmat1})
correspond, respectively, to the bilinear forms,
\begin{align}
\mathbf{v}^T\left(\boldsymbol{\Omega}_{d}\cdot\mathbf{G}\right)\mathbf{u} & =\sum_{\kappa \in \mathcal{E}}\int_{\kappa}\left(\boldsymbol{\Omega}_{d}\cdot\nabla_{\mathbf{x}}u\right)vd\mathbf{x},\label{eq:bilinear form for G}\\
\mathbf{v}^TF^{(d)}\mathbf{u} & =-\sum_{\Gamma\in\mathcal{F}}\int_{\Gamma}\boldsymbol{\Omega}_{d}\cdot\mathbf{n}\left\llbracket u\right\rrbracket \left\{ v\right\} dS+\frac{1}{2}\sum_{\Gamma\in\mathcal{F}}\int_{\Gamma}\left|\boldsymbol{\Omega}_{d}\cdot\mathbf{n}\right|\left\llbracket u\right\rrbracket \left\llbracket v\right\rrbracket dS,\label{eq:Fd bilinear form}\\
\mathbf{v}^TM_{t}\mathbf{u} & =\sum_{\kappa \in \mathcal{E}}\int_{\kappa}\sigma_{t}uvd\mathbf{x},\label{eq:bilinear form for Mt}\\
\mathbf{v}^TM_{a}\mathbf{u} & =\sum_{\kappa \in \mathcal{E}}\int_{\kappa}\sigma_{a}uvd\mathbf{x}.\label{eq:bilinear form for Ma}
\end{align}
Note that in our convention bold symbols indicate vectors and capital (from
the Latin alphabet) symbols indicate matrices. In addition, the notation $\mathbf{G}$ is shorthand for a vector with three matrix components,  $\mathbf{G} = \left( G_{1}, G_{2}, G_{3} \right)  $, so that 
$$
\boldsymbol{\Omega}_{d}\cdot\mathbf{G} = \sum_{j=1}^3  \left( \boldsymbol{\Omega}_{d} \right)_{j} G_j .
$$
\cblue{Also recall that each direction
$\boldsymbol{\Omega}_{d}$ has a corresponding reversed direction
$\boldsymbol{\Omega}_{d'}=-\boldsymbol{\Omega}_{d}$ with identical weight $w_{d'}=w_{d}$,
and note the useful identities, $\sum_{d}w_{d} =4\pi$,
$\sum_{d}w_{d}\boldsymbol{\Omega}_{d}\boldsymbol{\Omega}_{d}^{T}=\frac{4\pi}{3}I$,
$\sum_{d}w_{d}\boldsymbol{\Omega}_{d}=\mathbf{0}$, and
$\sum_{d}w_{d}\boldsymbol{\Omega}_{d}\left|\boldsymbol{\Omega}_{d}\cdot\mathbf{n}\right|=\mathbf{0}$.}

To reformulate equation (\ref{eq:Tmat1}),
define the column vector $\mathbf{\boldsymbol{\psi}=}\left(\boldsymbol{\psi}^{(1)};...;\boldsymbol{\psi}^{(N_{\Omega})}\right)$
and projection
\begin{equation}
 \left(P_{0}\boldsymbol{\psi}\right)^{(d)} = \frac{1}{4\pi}\sum_{d'}w_{d'}\boldsymbol{\psi}_{d'}= \frac{1}{4 \pi} \boldsymbol{\varphi},\,\,\,\,\,d=1,\ldots,N_{\Omega}.\label{eq:P0}
\end{equation}
$P_{0}$ is a weighted average over direction $d$ that
projects the average on to all vector blocks. In the matrix sense,
$P_{0}$ is an $NN_{\Omega}\times NN_{\Omega}$ operator,
where each block row takes the form $\frac{1}{4\pi}[w_{0}I_{N},w_{1}I_{N},...,w_{N_{\Omega}}I_{N}]$.
$P_{0}$ being a projection relies on the fact that $\sum_{d}w_{d}=4\pi$.
Defining 
\begin{equation} \label{eq:W-inner product}
W=\textnormal{diag}\left[w_{0}I_{N},w_{1}I_{N},...,w_{N_{\Omega}}I_{N}\right], \,\,\,\, \langle\mathbf{\mathbf{x}},\mathbf{y}\rangle_{W}=\langle W\mathbf{x},\mathbf{y}\rangle, 
\end{equation}
$P_{0}$ is an orthogonal projection in the $W$-inner product.
Letting $Q_{0}:=I-P_{0}$ denote the orthogonal complement to $P_{0}$,
recall that for any vector $\boldsymbol{\psi}$, $\|\boldsymbol{\psi}\|_{W}=\|P_{0}\boldsymbol{\psi}\|_{W}+\|Q_{0}\boldsymbol{\psi}\|_{W}$.
Now, rewrite \eqref{eq:Tmat1} as
\begin{equation}
\label{eq:Tmat2}
\left[I+\varepsilon M_{t}^{-1}\left(\boldsymbol{\Omega}_{d}\cdot\mathbf{G}+F^{(d)}\right)\right] \boldsymbol{\psi}^{(d)} -
  \frac{1}{4\pi}\left(I-\varepsilon^{2}M_{t}^{-1}M_{a}\right)\boldsymbol{\varphi} = 
  \frac{1}{4\pi}\varepsilon M_{t}^{-1}\left(\boldsymbol{q}_{\text{inc}}^{(d)}+\varepsilon\boldsymbol{q}^{(d)}\right).
\end{equation}
In matrix form, over all angles, the first term in \eqref{eq:Tmat2} operating on $\boldsymbol{\psi}^{(d)}$ is block diagonal in $d$, with each block corresponding
to a fixed direction $\boldsymbol{\Omega}_d$, and the second term a global angular coupling through projection $P_0$. A standard technique in transport
is to invert the first, block-diagonal term.
This approach corresponds to solving the linear transport equation
independently, for all directions $d$, and is known as a transport sweep.
Define $T_{\varepsilon}$ as the block-diagonal operator over direction $d$,
multiplied by $P_{0}$, when a transport sweep is applied:
\cred{
\[
T_{\varepsilon}= \frac{1}{4\pi} \textnormal{diag}_{d}\left[\left(I+\varepsilon M_{t}^{-1}\left(\boldsymbol{\Omega}_{d}\cdot\mathbf{G}+F^{(d)}\right)\right)^{-1}
  \left(I-\varepsilon^{2}M_{t}^{-1}M_{a}\right)\right]P_{0}.
\]
}
Then, equation (\ref{eq:Tmat1}) can be re-written as the preconditioned linear system 
\cred{
\begin{equation}
\label{eq:psi_d}
(I-T_{\varepsilon})\boldsymbol{\psi} = \tilde{\boldsymbol{q}} ,
\end{equation}
where 
\[
\tilde{\boldsymbol{q}}^{(d)}=\left(I+\varepsilon M_{t}^{-1}\left(\boldsymbol{\Omega}_{d}\cdot\mathbf{G}+F^{(d)}\right)\right)^{-1}\frac{1}{4\pi}\varepsilon M_{t}^{-1}\left(\boldsymbol{q}_{\text{inc}}^{(d)}+\varepsilon\boldsymbol{q}^{(d)}\right).
\]
}
Multiplying equation (\ref{eq:psi_d}) by the quadrature
weight, $w_{d}$, and summing over direction index, $d$, yields a
linear system for the scalar flux,
\begin{equation}
\label{eq:sweep linear system}
  \left(I-S_{\varepsilon}\right)\boldsymbol{\varphi} = \mathbf{s},
\end{equation}
where 
\begin{align}
S_{\varepsilon} & =\sum_{d}w_{d}\left(I+\varepsilon M_{t}^{-1}\left(\boldsymbol{\Omega}_{d}\cdot\mathbf{G}+F^{(d)}\right)\right)^{-1}\frac{1}{4\pi}\left(I-\varepsilon^{2}M_{t}^{-1}M_{a}\right),\label{eq:sweep matrix-2}   \\
\mathbf{s} & =\sum_{d}w_{d}\left(I+\varepsilon M_{t}^{-1}\left(\boldsymbol{\Omega}_{d}\cdot\mathbf{}+F^{(d)}\right)\right)^{-1}\varepsilon M_{t}^{-1} \varepsilon\frac{1}{4\pi} \left(\boldsymbol{q}_{\text{inc}}^{(d)}+\boldsymbol{q}^{(d)}\right).\label{eq:sweep right hand side-2}
\end{align}

We note that, in applying the operator
$T_{\varepsilon}$, we need to invert $\left[I+\varepsilon M_{t}^{-1}\left(\boldsymbol{\Omega}_{d} \cdot\mathbf{G}+F^{(d)}\right)\right]$. As it turns out, this is not always computationally tractable, particularly in the case of
HO curved meshes. Theorem \ref{th:mesh_cycles} and Section \ref{subsec:Mesh-cycles} analyze a more general case where
this term is not inverted exactly.

\begin{rem}
\cred{
Our analysis of equation \eqref{eq:psi_d} is valid under the assumption that 
$$ 
\varepsilon  \| M_{t}^{-1} \left(\boldsymbol{\Omega}_d \cdot \mathbf{G} + F^{(d)} \right)  \| < 1 ,\,\,\,\,\,\,\, \varepsilon^2  \| M_{t}^{-1}  M_{a} \| < 1 .
$$ Since  $ \| M_{t}^{-1} \left(\boldsymbol{\Omega}_d \cdot \mathbf{G} + F^{(d)} \right)  \|$ scales like $1/\left( \sigma_t h_{\mathbf{x}} \right)$, where $h_{\mathbf{x}}$ denotes the characteristic mesh spacing, the error bounds in Theorems \ref{th:MIP}-\ref{th:mesh_cycles} below are small as long as
$$ \eta = \min \left\{ \varepsilon/\left( h_{\mathbf{x}} \sigma_t \right), \varepsilon \sqrt{ \sigma_a/ \sigma_t} \right\}  \ll 1 .$$ The regime $\eta \ll 1$ corresponds to the standard optically thick limit.
}
\end{rem}

\subsection{Useful identities}

Next we present two identities that will be used regularly in further derivations.
\cblue{
First, applying integration by parts to the term
\[
  \mathbf{v}^T \boldsymbol{\Omega}_{d}\cdot\mathbf{G} \mathbf{u} =
  \sum_{\kappa \in \mathcal{E}} \int_{\kappa}
    \left(\boldsymbol{\Omega}_{d}\cdot\nabla_{\mathbf{x}}u\right)vd\mathbf{x} 
\]
in 
$\mathbf{v}^T \left( \boldsymbol{\Omega}_{d}\cdot
 \mathbf{G}+F^{(d)} \right) \mathbf{u}$
yields the identity 
\begin{equation}
\label{eq:relation between F and tildeF}
  \boldsymbol{\Omega}_{d}\cdot\mathbf{G}+F^{(d)}=
  -\boldsymbol{\Omega}_{d}\cdot\mathbf{G}^T+\tilde{F}^{(d)},
\end{equation}
where the matrix $\tilde{F}^{(d)}$ corresponding to the bilinear form 
\begin{equation}
\mathbf{v}^T\tilde{F}^{(d)}\mathbf{u}= -\sum_{\Gamma\in\mathcal{F}}\int_{\Gamma}\boldsymbol{\Omega}_{d}\cdot\mathbf{n}\left\{ u\right\} \left\llbracket v\right\rrbracket dS+\sum_{\Gamma\in\mathcal{F}}\int_{\Gamma}\frac{1}{2}\left|\boldsymbol{\Omega}_{d}\cdot\mathbf{n}\right|\left\llbracket u\right\rrbracket \left\llbracket v\right\rrbracket dS.\label{eq:tildeFd bilinear form}
\end{equation}
}

A second property follows immediately from equations (\ref{eq:Fd bilinear form})
and (\ref{eq:tildeFd bilinear form}). Let $P$ denote a projection
onto the space of continuous functions with zero boundary values.
Then, $P\mathbf{v}$ corresponds to a continuous function with zero
boundary value and, therefore, $\left\llbracket v\right\rrbracket =0$
on each interior mesh face $\Gamma$ and $v=0$ on each boundary face.
From expression (\ref{eq:Fd bilinear form}), we see that 
$
\left(P\mathbf{v}\right)^T\tilde{F}^{(d)}\mathbf{u}=\mathbf{v}^TP^T\tilde{F}^{(d)}\mathbf{u}=\mathbf{0},
$
for any $\mathbf{u}$ and $\mathbf{v}$. Since $\mathbf{u}$ and $\mathbf{v}$
are arbitrary, $P^T\tilde{F}^{(d)}=\mathbf{0}$. A similar
identity follows from expression (\ref{eq:tildeFd bilinear form}),
yielding the two identites
\begin{equation}
F^{(d)}P=\mathbf{0},\hspace{8ex}P^T\tilde{F}^{(d)}=\mathbf{0}.\label{eq:F*P and PT * tildeFd}
\end{equation}

\subsection{Need for preconditioning in the optically thick limit}

To motivate DSA and further analysis in this paper, we state the following Proposition which shows
that preconditioning the linear system in (\ref{eq:psi_d})
is important in the optically thick limit of small $\varepsilon$.
The proof of Proposition \ref{prop:condition number} is given in the Appendix. 

\begin{proposition}
\label{prop:condition number}Assume that the matrix $I-T_{\varepsilon}$
in the linear system (\ref{eq:psi_d}) is invertible.
Then the condition number of the matrix $I-T_{\varepsilon}$ from equation
(\ref{eq:psi_d}) satisfies
\[
\operatorname{cond}(I-T_{\varepsilon})=\|I-T_{\varepsilon}\|_W \|(I-T_{\varepsilon})^{-1}\| _W \geq\mathcal{O}\left(\varepsilon^{-2}\right) ,
\]
where the norm $\| \cdot \|_W$ is defined by equation  (\ref{eq:W-inner product}).
In addition, suppose that $E_{\varepsilon}$ inverts $P_{0}\left(I-T_{\varepsilon}\right)P_{0}$
on the range of $P_{0}$ to within $\mathcal{O}\left(\varepsilon\right)$,
$E_{\varepsilon}P_{0}\left(I-T_{\varepsilon}\right)P_{0}=P_{0}+\mathcal{O}\left(\varepsilon\right)$.
Then the preconditioned matrix $\left( \left( I-P_{0} \right) +  E_{\varepsilon}P_{0}\right)\left(I-T_{\varepsilon}\right)$
is an $\mathcal{O}(\varepsilon)$ perturbation of the identity,
\begin{equation}
\left( \left( I-P_{0} \right) +  E_{\varepsilon}P_{0}\right)\left(I-T_{\varepsilon}\right)=I+\mathcal{O}\left(\varepsilon\right).\label{eq:DSA system}
\end{equation}
\end{proposition}

\cred{
The relationship 
\begin{equation}
\left( P_{0} \left(I-T_{\varepsilon}\right) P_0 \boldsymbol{\psi} \right)^{(d)} =
  \left(I-S_{\varepsilon}\right) \frac{\varphi}{4\pi},
  \quad d = 1,\ldots, N_{\Omega},
\end{equation}
connects the Theorems in Section~\ref{subsec:Statement-of-theorems} with Proposition~\ref{prop:condition number}.
}

\section{DSA preconditioners for HO DG discretizations on curved meshes}

\subsection{Overview of the DSA preconditioners and statement of the theorems}

\label{subsec:Statement-of-theorems}

This section presents the main theoretical contributions of this paper, the proofs of which are contained in the following subsections. 

First we present results on a symmetric interior penalty (SIP) DSA
preconditioner. To do so, define the SIP DSA matrix,
\begin{equation}
D_{\varepsilon}=\frac{1}{\varepsilon}F_{0}+D_{0},\label{eq:Deps}
\end{equation}
where

\begin{equation}
D_{0}=\frac{1}{3}\mathbf{G}^{T}\cdot M_{t}^{-1}\mathbf{G}-\tilde{\mathbf{F}}_{1}\cdot M_{t}^{-1}\mathbf{G}+\mathbf{G}^T\cdot M_{t}^{-1}\mathbf{F}_{1}+M_{a},\label{eq:D}
\end{equation}
and

\begin{equation}
F_{0}=\frac{1}{4\pi}\sum_{d}w_{d}F^{(d)},\,\,\,\,\,\mathbf{F}_{1}=\frac{1}{4\pi}\sum_{d}w_{d}\boldsymbol{\Omega}_{d}F^{(d)},\,\,\,\,\tilde{\mathbf{F}}_{1}=\frac{1}{4\pi}\sum_{d}w_{d}\boldsymbol{\Omega}_{d}\tilde{F}^{(d)}.\label{eq:F_def}
\end{equation}
In the previous equations, $\mathbf{F}_{1}$ and  $\tilde{\mathbf{F}}_{1}$ correspond to vectors of matrices; for example, in three spatial dimensions
$$
   \left( \mathbf{F}_{1} \right)_j = \frac{1}{4\pi}\sum_{d}w_{d} \left( \boldsymbol{\Omega}_{d} \right)_j F^{(d)}, \,\,\,\,\, j = 1,2,3.
$$
Assuming that the mesh is first order and that the opacities, $\sigma_{t}$ and $\sigma_{a}$, are constants, it turns out (see Section~\ref{sec:Bilinear-form-for})
that $D_{\varepsilon}$ corresponds to the bilinear form, 
\begin{equation}
\mathbf{v}^TD_{\varepsilon}\mathbf{u}=\mathcal{B}_{\text{SIP}}\left(\cdot,\cdot\right),\label{eq:Deps and BMIP}
\end{equation}
where
\begin{equation}
\begin{split}
\mathcal{B}_{\text{SIP}}\left(u,v\right):= & \frac{1}{\varepsilon}\sum_{\Gamma\in\mathcal{F}}\int_{\Gamma}\alpha\left\llbracket u\right\rrbracket \left\llbracket v\right\rrbracket dS+\sum_{\kappa\in\mathcal{E}}\int_{\kappa}\frac{1}{3\sigma_{t}}\nabla_{\mathbf{x}}u\cdot\nabla_{\mathbf{x}}vd\mathbf{x}+\sum_{\kappa\in\mathcal{E}}\int_{\kappa}\sigma_{a}uvd\mathbf{x} - \\
 & \,\,\,\,\,\sum_{\Gamma\in\mathcal{F}}\int_{\Gamma}\left\llbracket u\right\rrbracket \left\{ \mathbf{n}\cdot\frac{1}{3\sigma_{t}}\nabla_{\mathbf{x}}v\right\} dS-\sum_{\Gamma\in\mathcal{F}}\int_{\Gamma}\left\llbracket v\right\rrbracket \left\{ \mathbf{n}\cdot\frac{1}{3\sigma_{t}}\nabla_{\mathbf{x}}u\right\} dS.\label{eq:BMIP}
\end{split}
\end{equation}
Here, the function $\alpha\left(\cdot\right)$ in the first integral
is defined as
\begin{equation}
\alpha\left(\mathbf{x}\right)= \frac{1}{4\pi} \sum_{d}w_{d}\left|\boldsymbol{\Omega}_{d}\cdot\mathbf{n}\left(\mathbf{x}\right)\right|,\,\,\,\,\mathbf{x}\in\Gamma\in\mathcal{F},\label{eq:alpha}
\end{equation}
and converges to $1/4$ in the limit of a large number of angles,
$\boldsymbol{\Omega}_{d}$. The bilinear form in (\ref{eq:BMIP})
corresponds to a variant of the symmetric interior penalty discretization
of the reaction-diffusion operator,
\[
\nabla_{\mathbf{x}}\cdot\left(\frac{1}{3\sigma_{t}}\nabla_{\mathbf{x}}\right)-\sigma_{a} . \label{eq:diffusion}
\]

Theorem~\ref{th:MIP} shows that preconditioning the fixed-point
iteration based on $(I-S_{\varepsilon})$ (\ref{eq:sweep linear system})
with the DSA matrix $D_{\varepsilon}$ results in fast convergence
in the optically thick limit. 
\begin{theorem}[SIP DSA preconditioner]
\label{th:MIP}Assume that the function $\alpha\left(\cdot\right)$
defined in equation (\ref{eq:alpha}) is uniformly bounded away from
zero on each interior and boundary mesh faces. Then
\[
\left(\varepsilon^{2}D_{\varepsilon}\right)^{-1}M_{t}\left(I-S_{\varepsilon}\right)=I+\mathcal{O}\left(\varepsilon\right).
\]
\end{theorem}
Theorem~\ref{th:MIP} states that the preconditioned iteration matrix
looks like the identity plus an $\mathcal{O}(\varepsilon)$ perturbation.
For small $\varepsilon$, this ensures a well-conditioned iteration
matrix and fast convergence. Under the assumptions of Theorem~\ref{th:MIP},
it follows from the identity (see Section \ref{sec:Bilinear-form-for})
\[
\mathbf{v}^TF_{0}\mathbf{u}=\sum_{\Gamma\in\mathcal{F}}\int_{\Gamma}\alpha\left\llbracket u\right\rrbracket \left\llbracket v\right\rrbracket dS
\]
that $F_{0}$ has a nullspace consisting of continuous functions with
zero boundary values. For example, if $\mathbf{u}$ is in the nullspace
of $F_{0}$, then
\[
\mathbf{u}^TF_{0}\mathbf{u}=\frac{1}{\varepsilon}\sum_{\Gamma\in\mathcal{F}}\int_{\Gamma}\alpha\left\llbracket u\right\rrbracket ^{2}dS=0,
\]
and so the jump $\left\llbracket u\right\rrbracket $ must vanish
on each interior mesh face and $u$ must vanish on each boundary face.
It follows that the condition number of $D_{\varepsilon}$ scales
like $\mathcal{O}\left(\varepsilon^{-1}\right)$, and a good preconditioner
is required to efficiently invert the interior penalty DSA matrix.
Unfortunately, HO DG discretizations can prove difficult
for fast linear solvers and preconditioners, such as algebraic multigrid
(AMG), even when considering elliptic problems \cite{Bastian:2012hk,Siefert:2014je,Prill:2009wg,Olson:2011ju}.
This difficulty is compounded on highly unstructured grids, which are some of the motivating problems here.

Fortunately, the proof of Theorem~\ref{th:MIP} also yields a better-conditioned
DSA preconditioner for optically thick problems. In fact, let $P$ denote an (arbitrary)
projection of functions in the DG space onto the continuous functions,
and $Q=I-P$ be its complement.
 Then Theorem~\ref{th:new_DSA}
develops a two-part, additive DSA matrix; a single DSA step involves three applications of
$P\left(P^TD_0P\right)^{-1}P^T$ (that is, solving a continuous Galerkin diffusion discretization), and one 
application of $Q\left(Q^TF_{0}Q\right)^{-1}Q^T$
(solving in the complement). In the optically thick limit, this DSA matrix is proven to have the same
theoretical iteration efficiency as the symmetric interior penalty DSA matrix
discussed in Theorem \ref{th:MIP}, and its application requires inverting matrices with condition
number independent of $\varepsilon$. 
\begin{theorem}
\label{th:new_DSA}Let $P$ denote a projection on to the subspace of $ \mathcal{U}$ containing continuous polynomials with zero boundary values, and let $Q=I-P$. Define
the operators
\[
E_{P}=P\left(P^TD_{0}P\right)^{-1}P^T,\hspace{5ex} E_{Q}=Q\left(Q^TF_{0}Q\right)^{-1}Q^T,
\]
and 
\[
E_{\varepsilon}=\frac{1}{\varepsilon}E_{P}+\left(I-E_{P}D_0\right)E_{Q}\left(I-D_0E_{P}\right),
\]
with $D_0$ as in \eqref{eq:D}. Then
\[
\frac{1}{\varepsilon}E_{\varepsilon}M_{t}\left(I-S_{\varepsilon}\right)=I+\mathcal{O}\left(\varepsilon\right).
\]
~
\end{theorem}
As in Theorem~\ref{th:MIP}, Theorem~\ref{th:new_DSA} proves that
the preconditioned operator is an $\mathcal{O}\left(\varepsilon\right)$
perturbation of the identity and is thus well-conditioned for small
$\varepsilon$, and the corresponding fixed-point iteration will converge
rapidly. Note that, using equation (\ref{eq:BMIP}) for the bilinear form corresponding to $P\left(P^TDP\right)^{-1}P^T$, it is straightforward to see that the matrix $P\left(P^TD_0 P\right)^{-1}P^T$ corresponds to solving a continuous Galerkin discretization of the diffusion equation (\ref{eq:diffusion}) (for constant opacities $\sigma_a$ and $\sigma_t$).
\cgreen{
In practice, the projection matrix $P$ is formed as a sparse matrix that
(i) interpolates the DG solution at the internal (per element) Gauss-Lobatto
nodes of the CG space, (ii) averages overlapping CG degrees of freedom on element
faces, and (iii) zeroes out CG degrees-of-freedom on each mesh boundary face.
Note that a number of works have considered preconditioning elliptic DG discretizations
with a projection onto continuous functions. To our knowledge, this was first considered in the
widely unrecognized paper by Warsa et al. \cite{warsa03}, and has been considered in a number of other
papers more recently \cite{omalley17,bastian12,Antonietti17,pazner2019efficient}.
Such approaches are similar in principle to Theorem \ref{th:new_DSA}, but here we directly
precondition the larger transport iteration by projecting onto continuous functions, rather
than trying to solve the DG DSA matrix with a continuous preconditioner. However, examples of
projections $P$ and $Q$ can be found in \cite{warsa03,omalley17,bastian12,Antonietti17,pazner2019efficient}.
}

The final result of this paper regards applying DSA to HO (curved)
meshes. In particular, consider the general linear system in \eqref{eq:Tmat2}, expressed as a single operator on $\boldsymbol{\psi}$:
\cred{
\begin{align}\label{eq:bigsys}
\left[ \left(I+\varepsilon H\right) -
\left(I-\varepsilon^{2}M_{t}^{-1}M_{a}\right)P_{0}\right] \boldsymbol{\psi} =
\mathbf{s} .
\end{align}
}
Often it is possible to order the mesh elements so that $H=\textnormal{diag}_{d}\left[M_{t}^{-1}\left(\boldsymbol{\Omega}_{d}\cdot\mathbf{G}+F^{(d)}\right)\right]$
is block lower triangular, with blocks corresponding to mesh elements.
In such cases, $I+\varepsilon H$ can be inverted directly to give
the equivalent (but better conditioned) system
\cred{
\begin{equation}
\label{eq:invH}
\left(I-T_{\varepsilon}\right) \boldsymbol{\psi}  =
\left(I+\varepsilon H\right)^{-1} \mathbf{s} ,
\end{equation}
where 
\[
T_{\varepsilon}=\left(I+\varepsilon H\right)^{-1}\left(I-\varepsilon^{2}M_{t}^{-1}M_{a}\right)P_{0}.
\]
}
However, for HO meshes, it is typically the case that $H$ is no longer block lower triangular, and cannot be
easily inverted through a forward solve. Recent work developed a nonsymmetric AMG algorithm that has
proved effective to invert HO DG transport discretizations on HO meshes \cite{air1,air2}, albeit with a
larger overhead cost compared with a forward solve. Alternatively, a graph-based algorithm was developed in \cite{sweep18}
to replace the inversion with a Gauss-Seidel type iteration in a pseudo-optimal ordering when mesh cycles are present.  
To consider an approximate inversion through an ordered Gauss-Seidel, suppose that we choose a mesh element ordering
that leads to a decomposition,
\[
H=H_{\leq}+H_{>},
\]
where 
\[
H_{\leq}=\textnormal{diag}_{d}\left[M_{t}^{-1}\left(\boldsymbol{\Omega}_{d}\cdot\mathbf{G}+F_{\leq}^{(d)}\right)\right],\,\,\,\,\,\,H_{>}=\textnormal{diag}_{d}\left[M_{t}^{-1}F_{>}^{(d)}\right].
\]
Here, we invert $H_{\leq}$ exactly and move $H_{>}$ to the right-hand side. For example, $H_{\leq}$
corresponds to the lower-triangular part of the matrix ordering in \cite{sweep18}, which is inverted in an
ordered Gauss-Seidel iteration. 

The following theorem shows that three transport sweeps with lagging---that is, three applications of $\left(I+\varepsilon H_{\leq}\right)^{-1}$---
yields an efficient preconditioner using the DSA matrix from Theorem~\ref{th:MIP}~or~\ref{th:new_DSA}.

\begin{theorem}\label{th:mesh_cycles}
Let $I - T_{\varepsilon}$ be the preconditioned linear system
in \eqref{eq:invH} that corresponds to applying $(I + \varepsilon H)^{-1}$ as a preconditioner. Define $I - \widetilde{T}_{\varepsilon}$
as the preconditioned linear system associated with applying three iterations of $(I + \varepsilon H_{\leq})^{-1}$ to \eqref{eq:bigsys}, while
keeping the term $\left(I-\varepsilon^{2}M_{t}^{-1}M_{a}\right)P_{0}\boldsymbol{\psi}$ fixed. Then
\[
\widetilde{T}_{\varepsilon}=T_{\varepsilon}+\mathcal{O}\left(\varepsilon^{3}\right),
\]
and, letting $E_{\varepsilon}$ correspond to the DSA preconditioner in Theorem~\ref{th:MIP}~or~\ref{th:new_DSA},
\[
E_{\varepsilon}P_{0}\left(I-T_{\varepsilon}\right)P_{0}=E_{\varepsilon}P_{0}\left(I-\widetilde{T}_{\varepsilon}\right)P_{0}+\mathcal{O}\left(\varepsilon\right).
\]
\end{theorem}

\begin{rem}
Note that moving the term $\left(I-\varepsilon^{2}M_{t}^{-1}M_{a}\right)P_{0}\boldsymbol{\psi}$ in the linear system to
the right-hand side and fixing it---that is, not updating $\left(I-\varepsilon^{2}M_{t}^{-1}M_{a}\right)P_{0}\boldsymbol{\psi}$
based on an updated $\boldsymbol{\psi}$---is not typical in a fixed-point iterative method. One can also work out the
error-propagation matrix for multiple iterations that include updating this term each iteration. For this variant, the
asymptotics in $\varepsilon$ do not clearly indicate a well-conditioned system for $\varepsilon \ll 0$, as obtained in
Theorem \ref{th:mesh_cycles}. However, numerically, updating $\left(I-\varepsilon^{2}M_{t}^{-1}M_{a}\right)P_{0}\boldsymbol{\psi}$
each iteration proves to be more robust for larger $\varepsilon$, which is discussed in Section \ref{sec:numerical}.
\end{rem}

\begin{rem}
When $\varepsilon \gtrapprox h_x \sigma_t \ll 1$, the preconditioned matrix $\left(\varepsilon^{2}D_{\varepsilon}\right)^{-1}M_{t}\left(I-S_{\varepsilon}\right)$ from 
Theorem~\ref{th:MIP} becomes ill-conditioned (since the spectrum of $D_{\varepsilon}^{-1}$ goes to zero for high-frequency eigenvectors of $D_{\varepsilon}$). In practice, 
we therefore use the preconditioner 
\begin{equation}\label{eq:modprec}
I  + \left(\varepsilon^{2} D_{\varepsilon}\right)^{-1}M_t.
\end{equation}
In the thick limit, Theorem~\ref{th:MIP} yields $\left(\varepsilon^{2} D_{\varepsilon}\right)^{-1}M_t
\left(I-S_{\varepsilon}\right)=I+\mathcal{O}\left(\varepsilon\right)$. In addition,
$\|I-S_{\varepsilon}\|\sim\mathcal{O}(1)$ and is well-conditioned in the thin limit, and thus the resulting
preconditioned matrix using \eqref{eq:modprec} has the same asymptotic efficiency for optically thick problems,
and is well-conditioned for optically thin problems.
\end{rem}

\cred{
\begin{rem}
As noted previously, the symmetric interior penalty DSA matrix in Theorem~\ref{th:MIP} has a penalty parameter that scales like $\varepsilon^{-1}$, and this can present challenges for standard algebraic multigrid preconditioners. However, the recent preconditioner developed in \cite{Antonietti-Sarti-Verani-Zikatanov-2017}  utilizes a decomposition of the DG space into a continuous component and a correction. The resulting preconditioner involves solving a continuous Galerkin diffusion matrix, and a correction step involving a Jacobi iteration. In this way, the method has a close connection to the preconditioner in \cite{O'Malley_Kophazi_2017}. The preconditioner in \cite{Antonietti-Sarti-Verani-Zikatanov-2017} results in a number of iterations on the preconditioned interior penalty DSA matrix that is provably independent of the local DG polynomial order, the mesh spacing, and the penalty parameter. Therefore, the interior penalty DSA matrix from Theorem~\ref{th:MIP}, in conjunction with the preconditioner in  \cite{Antonietti-Sarti-Verani-Zikatanov-2017}, can serve as an effective alternative to the DSA matrix in Theorem~\ref{th:new_DSA}. 
\end{rem}
}

Proofs of Theorems \ref{th:MIP}-\ref{th:mesh_cycles} are given
in Section \ref{sec:Proofs}.

\subsection{Connection to previous work} \label{subsec:Connection to previous work}

\subsubsection{The modified interior penalty DSA preconditioner}

We first connect our derivation and analysis of the SIP DSA preconditioner to the MIP DSA preconditioner in \cite{Wang-Ragusa-2010}, and then relate the SIP DSA preconditioner to the consistent DSA preconditioner derived in \cite{Warsa-Wareing-Morel-2002} for linear DG discretizations. 

In \cite{Wang-Ragusa-2010} the authors numerically demonstrate that using the modified interior penalty (MIP) DSA matrix yields uniformly good convergence in both optically thick and thin regimes. The corresponding bilinear form is similar to equation (\ref{eq:BMIP}), but the penalty coefficient $\gamma$ in the penalty term, 
$$ \sum_{\Gamma\in\mathcal{F}}\int_{\Gamma} \gamma \left\llbracket u\right\rrbracket \left\llbracket v\right\rrbracket dS , $$
is modified outside of the optically thick limit. In particular, letting $h_{\mathbf{x}}$ denote the characteristic mesh spacing, the MIP penalty coefficient  $ \gamma $ in \cite{Wang-Ragusa-2010} scales like $ \max \left( 1/\left( 4 \varepsilon \right) ,  C_p/ \left( \sigma_t h_{\mathbf{x}} \right) \right)  $, where $C_p$ is a constant that depends on the finite element local polynomial order. Notice that, when 
$\varepsilon \lessapprox \sigma_t h_{\mathbf{x}} $, the MIP penalty coefficient reduces to $ 1/(4 \varepsilon) \approx \alpha/\varepsilon$ (this inequality becomes an equality in the limit of an infinite number of quadrature angles), which is identical to the SIP DSA penalty coefficient in  equation (\ref{eq:BMIP}). Therefore, Theorem~\ref{th:MIP} justifies the numerically observed behavior in  \cite{Wang-Ragusa-2010} when $\varepsilon \lessapprox \sigma_t h_{\mathbf{x}} $. 

When $\varepsilon \gtrapprox \sigma_t h_{\mathbf{x}} $, the analysis in Theorem~\ref{th:MIP} breaks down. Nevertheless, at this point the mesh spacing $h_\mathbf{x}$ is small enough to numerically resolve the continuum transport equation (\ref{eq:linear SN transport}). It is then expected that the analysis of DSA acceleration for the continuum S$_N$ transport equation using the continuum diffusion equation can describe this situation
(for example, see \cite{Larsen_1984}). In particular, as long as the discrete diffusion equation remains a valid discretization of the continuum diffusion equation when $\varepsilon \gtrapprox \sigma_t h_{\mathbf{x}} $, we expect rapid acceleration for both optically thick and thin regimes. However, it is well-known
that the penalty parameter must be at least as large as $\mathcal{O}\left( \frac{1}{\sigma_t h_{\mathbf{x}}} \right)$ in order for the SIP discretization to remain a stable discretization of the continuum diffusion equation (for example, see \cite{Arnold_Brezzi_Cockburn_Marini_2002}). This motivates choosing $\kappa$ to scale like $ \max \left\{ 1/\left( 4 \varepsilon \right) ,  C_p/ \left( \sigma_t h_{\mathbf{x}} \right) \right\}  $ to ensure that the MIP DSA matrix both approximates the near-nullspace in the optically thick (ill-conditioned) limit $\varepsilon \lessapprox \sigma_t h_{\mathbf{x}} $, and also remains a good approximation to the continuum diffusion equation as $\varepsilon \gtrapprox \sigma_t h_{\mathbf{x}} $ and $h_{\mathbf{x}}$ begins to resolve the mean free path. 

\subsubsection{The nonsymmetric interior penalty DSA preconditioner}
\label{sec:consistent discretization}

In Section \ref{sec:numerical}, the SIP DSA preconditioner is shown to be robust for $\varepsilon \ll 1$, but does not converge for
moderate $\varepsilon$ (relative to the characteristic mesh spacing). Consider the nonsymmetric interior penalty (IP) version of the DSA matrix
\begin{equation}
\frac{1}{\varepsilon}F_{0}+\frac{1}{3}\boldsymbol{G}^{T}M_{t}^{-1}\boldsymbol{G}-\tilde{\mathbf{F}}_{1}\cdot M_{t}^{-1}\mathbf{G}+M_{a},
  \label{eq:modified DSA matrix}
\end{equation}
where we have neglected the term $\mathbf{G}^T\cdot M_{t}^{-1}\mathbf{F}_{1}$
from the symmetric interior penalty DSA preconditioner defining $D_{\varepsilon}$
(see \eqref{eq:D}). Dropping this term results in a nonsymmetric interior penalty (IP) 
discretization of the diffusion equation when the opacities are constant, and we observe empirically
that uniformly good convergence is obtained for all tested values of $\varepsilon$ using 
this DSA matrix (\ref{eq:modified DSA matrix}). In fact, for linear DG discretizations and straight-edged meshes, the nonsymmetric interior penalty DSA matrix (\ref{eq:modified DSA matrix}) reduces to the Warsa-Wareing-Morel consistent diffusion discretization  \cite{Warsa-Wareing-Morel-2002}. 

Also, a straightforward (but tedious) calculation shows that one can obtain the SIP DSA preconditioner by taking the first two (discrete) angular moments
of the discrete equation (\ref{eq:Tmat1}), and employing the following discrete version
of Fick's law 
\begin{equation}
\boldsymbol{\psi}^{(d)}=\frac{1}{4\pi}\boldsymbol{\varphi} - \varepsilon\frac{1}{4\pi}M_{t}^{-1}\left(\boldsymbol{\Omega}_{d}\cdot\mathbf{G}+F^{(d)}\right)\boldsymbol{\varphi}+\mathcal{O}\left(\varepsilon^{2}\right) . \label{eq:asymptotic Fick's law}
\end{equation}
Equation
(\ref{eq:asymptotic Fick's law}) results from equation (\ref{eq:Tmat1}),
\begin{align*}
\boldsymbol{\psi}^{(d)} & =\left(I+\varepsilon M_{t}^{-1}\left(\boldsymbol{\Omega}_{d}\cdot\mathbf{G}+F^{(d)}\right)\right)^{-1}\left(\frac{1}{4\pi}\boldsymbol{\varphi}+\varepsilon\frac{1}{4\pi}\boldsymbol{q}_{\text{inc}}^{(d)}\right)\\
 & =\frac{1}{4\pi}\boldsymbol{\varphi}-\varepsilon M_{t}^{-1}\left(\boldsymbol{\Omega}_{d}\cdot\mathbf{G}+F^{(d)}\right)\frac{1}{4\pi}\boldsymbol{\varphi}+\varepsilon\frac{1}{4\pi}\boldsymbol{q}_{\text{inc}}^{(d)}+\mathcal{O}\left(\varepsilon^{2}\right),
\end{align*}
where the constant vector $\varepsilon\left(4\pi\right)^{-1}\boldsymbol{q}_{\text{inc}}^{(d)}$
is neglected for simplicity since it only contributes to the right-hand
side.  Similarly, by instead employing
the following modified version of Fick's law in the discrete moment
equations, 
\begin{equation}
\boldsymbol{\psi}^{(d)}\approx\frac{1}{4\pi}\boldsymbol{\varphi}+\varepsilon\frac{1}{4\pi}M_{t}^{-1}\left(\boldsymbol{\Omega}_{d}\cdot\mathbf{J}\right),\label{eq:consistent Fick's law}
\end{equation}
an analogous calculation shows that that $\mathbf{J}=-\mathbf{G}\boldsymbol{\varphi}$, and leads to the nonsymmetric interior penalty DSA matrix (\ref{eq:modified DSA matrix}). 
In particular, the modified version of Fick's law (\ref{eq:consistent Fick's law})
results from neglecting the term $\varepsilon\left(4\pi\right)^{-1}M_{t}^{-1}F^{(d)}\boldsymbol{\varphi}$
in equation (\ref{eq:asymptotic Fick's law}). 

\cred{
\begin{rem}
The consistent P1 formulation is usually written in terms of both the current $\mathbf{J}$ and the scalar flux $\boldsymbol{\varphi}$. However, upon using the $1^{\text{st}}$ moment equation to express the discrete current in terms of the scalar flux and plugging the resulting expression in to the $0^{\text{th}}$  moment equation, one obtains an equation for the scalar flux only and this equation exactly corresponds to the non-symmetric interior penalty DSA preconditioner.
\end{rem}
}

\section{Proofs of main results\label{sec:Proofs}}

\subsection{Proofs of the theorems}

\label{sec:proofs}

We first establish the following lemma.

\begin{lemma}
\label{prop:singular matrix perturbation}
Consider the matrix 
\[
\widehat{D}=F_{0}+\varepsilon D,
\]
where $F_{0}$ is a symmetric, singular matrix. Define $P$ as a projection on to the nullspace of $F_{0}$, let $Q=I-P$
denote its complement, and define
\[
E_{P}=P\left(P^TDP\right)^{-1}P^T,\,\,\,\,\,E_{Q}=Q\left(Q^TF_{0}Q\right)^{-1}Q^T.
\]
Then,
\begin{align}
\widehat{D}^{-1} & =\frac{1}{\varepsilon}E_{P}+\left(I-E_{P}D\right)E_{Q}\left(I-DE_{P}\right)+\varepsilon \left(I-E_{P}D\right)R_{\varepsilon}\left(I-DE_{P}\right),\label{eq:asymptotic expansion for inverse of A}
\end{align}
where 
\[
R_{\varepsilon}=  \varepsilon\Big(I+\varepsilon E_{Q}\left(D-DE_{P}D\right)Q\Big)^{-1}E_{Q}\left(I-DE_{P}\right)DE_Q.
\]
In addition, suppose that $D=D_{0}+D_{1}$, where $P^TD_{1}=D_{1}P=\mathbf{0}$.
Then,
\begin{equation}
\left(F_{0}+\varepsilon D\right)^{-1}=\left(F_{0}+\varepsilon D_{0}\right)^{-1}+\mathcal{O}\left(\varepsilon\right).\label{eq:F0 +eps*( D0+D1 ) inverse}
\end{equation}
\end{lemma}
\begin{proof}
Consider the equation $\widehat{D}\mathbf{x}=\mathbf{y}$, and let $P$ be a projection onto the null space of $F_0$, and
$Q = I - P$ its complement. Similar to the proof of Proposition \ref{prop:condition number},
$\widehat{D}\mathbf{x}=\mathbf{y}$ can be expanded based on $P$ and $Q$ as a $2\times2$ system. First, note that
$\widehat{D}\mathbf{x}=\mathbf{y}$ can be written as
\begin{align*}
\widehat{D}\begin{pmatrix}P & Q\end{pmatrix}\begin{pmatrix}P\\ Q \end{pmatrix}\mathbf{x} & = (P + Q)\mathbf{y}.
\end{align*}
Now, we can multiply on the left by the full column-rank operator $\begin{pmatrix}P^T \\ Q^T\end{pmatrix}$ to yield
\begin{align*}
\begin{pmatrix}P^T \\ Q^T \end{pmatrix}\widehat{D}\begin{pmatrix}P & Q\end{pmatrix}\begin{pmatrix}P\\ Q \end{pmatrix}\mathbf{x} & =
  \begin{pmatrix}P^T \\ Q^T\end{pmatrix}(P + Q)\mathbf{y}, \\
\begin{pmatrix} P^T\widehat{D}P & P^T\widehat{D}Q \\ Q^T\widehat{D}P & Q^T\widehat{D}Q \end{pmatrix} 
  \begin{pmatrix} P\mathbf{x}\\ Q\mathbf{x} \end{pmatrix} & = \begin{pmatrix}P^T\mathbf{y} \\ Q^T \mathbf{y} \end{pmatrix}.
\end{align*}
Denote $\mathbf{x}_P = P\mathbf{x}$ and $\mathbf{x}_Q:=Q\mathbf{x}$. Using the equations $P^TF_0 = F_0P = \mathbf{0}$,
we can rewrite the linear system as
\begin{equation}
\begin{pmatrix} \varepsilon P^TDP & \varepsilon P^TDQ \\ \varepsilon Q^TDP & Q^T\widehat{D}Q \end{pmatrix} 
  \begin{pmatrix} \mathbf{x}_P \\ \mathbf{x}_Q \end{pmatrix} = \begin{pmatrix}P^T\mathbf{y} \\ Q^T \mathbf{y} \end{pmatrix}.
  \label{eq:first equation, sing perturb}
\end{equation}
Then,
\begin{align*}
\mathbf{x}_{P} & = \frac{1}{\varepsilon}\left(P^TDP\right)^{-1}P^T\mathbf{y}-\left(P^TDP\right)^{-1}\left(P^TDQ\right)\mathbf{x}_{Q}\\
 & =\frac{1}{\varepsilon}E_{P}\mathbf{y}-E_{P}DQ\mathbf{x}_{Q},
\end{align*}
where the second equality follows from noting that $\mathbf{x}_{P}=P\mathbf{x}$ and multiplying both sides by $P$.
Equation \eqref{eq:first equation, sing perturb} also yields
\begin{align*}
Q^T\widehat{D}Q\mathbf{x}_Q + \varepsilon Q^TDP \mathbf{x}_P & = Q^T \mathbf{y} , \\
\Big(Q^T\widehat{D}Q-\varepsilon\left(Q^TDP\Big)\left(E_{P}DQ\right)\right)\mathbf{x}_{Q} & =Q^T\mathbf{y}-\left(Q^TDE_{P}\right)\mathbf{y}, \\
\Big(Q^TF_{0}Q+\varepsilon Q^T\left(D-DE_{P}D\right)Q\Big)\mathbf{x}_{Q} & =Q^T\mathbf{y}-\left(Q^TDE_{P}\right)\mathbf{y}.
\end{align*}
Now, since the matrix $Q^TF_{0}Q$ above is invertible on the range of $Q^{\text{T}}$,
we can apply $\left(Q^TF_{0}Q\right)^{-1}Q^T$ to both sides to get
\begin{equation}
\left(I+\varepsilon Q\left(Q^TF_{0}Q\right)^{-1}Q^T\left(D-DE_{P}D\right)Q\right)\mathbf{x}_{Q}=
  \left(Q^TF_{0}Q\right)^{-1}Q^T\left(I-DE_{P}\right)\mathbf{y}.\label{eq:xQ num 2}
\end{equation}
Substituting $E_Q = Q\left(Q^TF_{0}Q\right)^{-1}Q^T$ in the left-hand side and applying the matrix identity
$\left(I+A\right)^{-1}=I-\left(I+A\right)^{-1}A$ to $\Big(I+\varepsilon E_{Q}\left(D-DE_{P}D\right)Q\Big)^{-1}$
yields
\begin{equation}
\Big(I+\varepsilon E_{Q}\left(D-DE_{P}D\right)Q\Big)^{-1}\left(Q^TF_{0}Q\right)^{-1}Q^T
  =\left(Q^TF_{0}Q\right)^{-1}Q^T- \tilde{R}_{\varepsilon},\label{eq:xQ with Reps}
\end{equation}
where 
\begin{align*}
\tilde{R}_{\varepsilon} & = \varepsilon\Big(I+\varepsilon E_{Q}\left(D-DE_{P}D\right)Q\Big)^{-1}E_{Q}\left(D-DE_{P}D\right)Q\left(\left(Q^TF_{0}Q\right)^{-1}Q^T\right) \\
&  = \varepsilon\Big(I+\varepsilon E_{Q}\left(D-DE_{P}D\right)Q\Big)^{-1}E_{Q}\left(I-DE_{P}\right)DE_Q.
\end{align*}
Therefore, using equation (\ref{eq:xQ with Reps}) in equation (\ref{eq:xQ num 2}), and noting that $Q\mathbf{x}_{Q}=\mathbf{x}_{Q}$,
\begin{align*}
\mathbf{x}_{Q} & =\left(Q^TF_{0}Q\right)^{-1}Q^T\left(I-DE_{P}\right)\mathbf{y} - \tilde{R}_{\varepsilon}\left(I-DE_{P}\right)\mathbf{y}, \\
 & =E_{Q}\left(I-DE_{P}\right)\mathbf{y} + \tilde{R}_{\varepsilon}\left(I-DE_{P}\right)\mathbf{y}.
\end{align*}
Solving for $\widehat{D}^{-1}\mathbf{y}$ yields the final result
\begin{align*}
\widehat{D}^{-1}\mathbf{y} & =\mathbf{x}_{P}+\mathbf{x}_{Q}\\
 & =\frac{1}{\varepsilon}E_{P}\mathbf{y} + (I -E_{P}DQ)\mathbf{x}_{Q} \\
 & =\frac{1}{\varepsilon}E_{P}\mathbf{y}+(I-E_{P}DQ)E_{Q}\left(I-DE_{P}\right)\mathbf{y} + 
  (I-E_{P}D)\tilde{R}_{\varepsilon}\left(I-DE_{P}\right).
\end{align*}
Equation \eqref{eq:F0 +eps*( D0+D1 ) inverse} follows by noting that 
\[
(I-E_{P}DQ)E_{Q}\left(I-DE_{P}\right)\mathbf{y} = (I-E_{P}D_0Q)E_{Q}\left(I-D_0E_{P}\right)\mathbf{y} ,
\]
and that terms involving $D_1$ only come up in the $\mathcal{O}(\varepsilon)$ remainder term, $\tilde{R}_{\varepsilon}$.
\end{proof}

We can now use Lemma~\ref{prop:singular matrix perturbation} to
prove Theorems~\ref{th:MIP}~and~\ref{th:new_DSA}
\begin{proof}[Proof of Theorem \ref{th:MIP}]
 Recall the definition from Lemma \ref{prop:singular matrix perturbation},
$H^{(d)}=M_{t}^{-1}\left(\boldsymbol{\Omega}_{d}\cdot\mathbf{G}+F^{(d)}\right),$
and the identities $\sum_{d}w_{d}=4\pi$, $\sum_{d}w_{d}\boldsymbol{\Omega}_{d}=\boldsymbol{0}$
and $F_{0}=\frac{1}{4\pi}\sum_{d}w_{d}F^{(d)}$. Using the identiy
for $\left(I+\varepsilon H^{(d)}\right)^{-1}$ in (\ref{eq:IplusH_inv}),
$I-S_{\varepsilon}$ can be expanded as
\begin{equation}
\begin{split}
I-S_{\varepsilon} & =I-\frac{1}{4\pi}\sum_{d}w_{d}\left(I+\varepsilon H^{(d)}\right)^{-1}\left(I-\varepsilon^{2}M_{t}^{-1}M_{a}\right) \\
 & =I-\frac{1}{4\pi}\sum_{d}w_{d}\left(I-\varepsilon H^{(d)}+\varepsilon^{2}\left(\left(H^{(d)}\right)^{2}-M_{t}^{-1}M_{a}\right)\right)+\mathcal{O}\left(\varepsilon^{3}\right) \\
 & =\varepsilon\left[\frac{1}{4\pi}\sum_{d}w_{d}H^{(d)}-\frac{1}{4\pi}\varepsilon\sum_{d}w_{d}\left(\left(H^{(d)}\right)^{2}-M_{t}^{-1}M_{a}\right)\right]+\mathcal{O}\left(\varepsilon^{3}\right) \\
 & =\varepsilon M_{t}^{-1}\left(F_{0}+\varepsilon\tilde{D}_{\varepsilon}\right)+\mathcal{O}\left(\varepsilon^{3}\right),\label{eq:I - S_eps}
\end{split}
\end{equation}
where $\tilde{D}_{\varepsilon}$ corresponds to the latter term in
(\ref{eq:I - S_eps}) and is given by 
\begin{align}
\tilde{D}_{\varepsilon} & =-\frac{1}{4\pi}M_{t}\sum_{d}w_{d}\left(\left(M_{t}^{-1}\left(\boldsymbol{\Omega}_{d}\cdot\mathbf{G}+F^{(d)}\right)\right)^{2}-M_{t}^{-1}M_{a}\right).\label{eq:matrix for D}
\end{align}
Recall the identity from (\ref{eq:relation between F and tildeF}),
$\boldsymbol{\Omega}_{d}\cdot\mathbf{G}+F^{(d)}=-\boldsymbol{\Omega}_{d}\cdot\mathbf{G}^T+\tilde{F}^{(d)}$,
the definitions of $\mathbf{F}_{1}=\frac{1}{4\pi}\sum_{d}w_{d}\boldsymbol{\Omega}_{d}F^{(d)}$
and $\tilde{\mathbf{F}}_{1}=\frac{1}{4\pi}\sum_{d}w_{d}\boldsymbol{\Omega}_{d}\tilde{F}^{(d)}$
from equation (\ref{eq:F_def}), and also that because $\boldsymbol{\Omega}_{d}$
is a scalar vector, it commutes in a certain sense; for example, 
\begin{align*}
\boldsymbol{\Omega}_{d}\cdot\boldsymbol{G}^{T}M_{t}^{-1}F^{(d)} & =\left(\Omega_{d_{1}},\Omega_{d_{2}},\Omega_{d_{3}}\right)\cdot\left(\mathbf{G}_{1}^{T},\mathbf{G}_{2}^{T},\mathbf{G}_{3}^{T}\right)M_{t}^{-1}F^{(d)}\\
 & =\left[\left(\mathbf{G}_{1}^{T},\mathbf{G}_{2}^{T},\mathbf{G}_{3}^{T}\right)M_{t}^{-1}F^{(d)}\right]\cdot\left(\Omega_{d_{1}},\Omega_{d_{2}},\Omega_{d_{3}}\right)\\
 & =\boldsymbol{G}^{T}M_{t}^{-1}\cdot\left(F^{(d)}\boldsymbol{\Omega}_{d}\right).
\end{align*}
Also recall the outer product summation $\sum_{d}w_{d}\boldsymbol{\Omega}_{d}\boldsymbol{\Omega}_{d}^{T}=\frac{4\pi}{3}I$.
Expanding the quadratic term in $\tilde{D}_{\varepsilon}$ and plugging
in these identities yields
\begin{align*}
\tilde{D}_{\varepsilon} & =-\frac{1}{4\pi}\sum_{d}w_{d}\bigg(-\boldsymbol{\Omega}_{d}\cdot\boldsymbol{G}^{T}M_{t}^{-1}\boldsymbol{\Omega}_{d}\cdot\boldsymbol{G}+\tilde{F}^{(d)}M_{t}^{-1}\boldsymbol{\Omega}_{d}\cdot\boldsymbol{G}-\\
 & \qquad\qquad\boldsymbol{\Omega}_{d}\cdot\boldsymbol{G}^{T}M_{t}^{-1}F^{(d)}+\tilde{F}^{(d)}M_{t}^{-1}F^{(d)}-M_{a}\bigg)\\
 & =\frac{1}{3}\boldsymbol{G}^{T}M_{t}^{-1}\boldsymbol{G}-\tilde{\mathbf{F}}_{1}\cdot M_{t}^{-1}\mathbf{G}+\mathbf{G}^T\cdot M_{t}^{-1}\mathbf{F}_{1}+M_{a}-\\
 & \qquad\qquad\frac{1}{4\pi}\sum_{d}w_{d}\tilde{F}^{(d)}M_{t}^{-1}F^{(d)}.
\end{align*}
Decompose $\tilde{D}_{\varepsilon}=D_{0}+D_{1}$,
where 
\begin{align}
D_{0} & =\left(\frac{1}{3}\mathbf{G}^{T}\cdot M_{t}^{-1}\mathbf{G}-\tilde{\mathbf{F}}_{1}\cdot M_{t}^{-1}\mathbf{G}+\mathbf{G}^T\cdot M_{t}^{-1}\mathbf{F}_{1}+M_{a}\right),\label{eq:D0}\\
D_{1} & =-\sum_{d}w_{d}\tilde{F}^{(d)}M_{t}^{-1}F^{(d)}.\nonumber 
\end{align}
Then,
\begin{align}
I-S_{\varepsilon} & =\varepsilon M_{t}^{-1}\left(F_{0}+\varepsilon\left(D_{0}+D_{1}\right)\right)+\mathcal{O}\left(\varepsilon^{3}\right).\label{eq:S_precond}
\end{align}
In the right-hand side of equation (\ref{eq:S_precond}), the lower-order
terms in $\varepsilon$ exactly take the form of the operator in Lemma
\ref{prop:singular matrix perturbation}, where $PD_{1}=D_{1}P=\mathbf{0}$.
To that end, from equation (\ref{eq:F0 +eps*( D0+D1 ) inverse}) in
Lemma \ref{prop:singular matrix perturbation},
\begin{equation}
\left(F_{0}+\varepsilon\left(D_{0}+D_{1}\right)\right)^{-1}=\frac{1}{\varepsilon}\left(\frac{1}{\varepsilon}F_{0}+D_{0}\right)^{-1}+\mathcal{O}(\varepsilon).\label{eq:F0D0_id}
\end{equation}
Defining $D_{\varepsilon}=\frac{1}{\varepsilon}F_{0}+D_{0}$, observe
from equations (\ref{eq:S_precond}) and (\ref{eq:F0D0_id}) that
\begin{align*}
(\varepsilon^{2}D_{\varepsilon})^{-1}M_{t}(I-S_{\varepsilon}) & =\frac{1}{\varepsilon}\left(\frac{1}{\varepsilon}F_{0}+D_{0}\right)^{-1}\left(F_{0}+\varepsilon\left(D_{0}+D_{1}\right)\right)+\mathcal{O}\left(\varepsilon\right)\\
 & =I+\mathcal{O}\left(\varepsilon\right),
\end{align*}
which completes the proof.
\end{proof}
\begin{proof}[Proof of Theorem \ref{th:new_DSA}]
 The proof of Theorem \ref{th:new_DSA} follows naturally from that
of Theorem \ref{th:MIP} and Lemma \ref{prop:singular matrix perturbation}.
From Lemma \ref{prop:singular matrix perturbation},
\[
(F_{0}+\varepsilon D_{0})^{-1}=\frac{1}{\varepsilon}E_{P}+\left(I-E_{P}D_0\right)E_{Q}\left(I-D_0E_{P}\right),
\]
where $E_{P}=P\left(P^TD_{0}P\right)^{-1}P^T$ and
$E_{Q}=Q\left(Q^TF_{0}Q\right)^{-1}Q^T$. Defining
$E_{\varepsilon}=(F_{0}+\varepsilon D_{0})^{-1}$
and appealing to (\ref{eq:S_precond}) and (\ref{eq:F0D0_id}), observe
that
\begin{align*}
\frac{1}{\varepsilon}E_{\varepsilon}M_{t}(I-S_{\varepsilon}) & =(F_{0}+\varepsilon D_{0})^{-1}\left[\left(F_{0}+\varepsilon\left(D_{0}+D_{1}\right)\right)+\mathcal{O}(\varepsilon^{2})\right]\\
 & =(F_{0}+\varepsilon D_{0})^{-1}\left(F_{0}+\varepsilon\left(D_{0}+D_{1}\right)\right)+\mathcal{O}\left(\varepsilon\right)\\
 & =I+\mathcal{O}(\varepsilon).
\end{align*}
\end{proof}

\subsection{Bilinear form for DG near-nullspace\label{sec:Bilinear-form-for}}

This section proves the identity (\ref{eq:Deps and BMIP}) relating
the DSA matrix (\ref{eq:Deps}) to the symmetric interior penalty
bilinear form (\ref{eq:BMIP}). In matrix form, (\ref{eq:Deps}) corresponds
to the bilinear form

\begin{equation}
\mathbf{v}^T\left(\frac{1}{\varepsilon}F_{0}+\frac{1}{3}\boldsymbol{G}^{T}M_{t}^{-1}\boldsymbol{G}-\tilde{\mathbf{F}}_{1}\cdot M_{t}^{-1}\mathbf{G}+\mathbf{G}^T\cdot M_{t}^{-1}\mathbf{F}_{1}+M_{a}\right)\mathbf{u}.\label{eq:bilinear form equivalent}
\end{equation}
Several of the relations are straightforward. The term $\mathbf{v}^{T}M_{a}\mathbf{u}=\sum_{\kappa\in\mathcal{E}}\int_{\kappa}\sigma_{a}uvd\mathbf{x}$
follows immediately from (\ref{eq:bilinear form for Ma}). Recalling
the definitions of $\alpha(\mathbf{x})$ (\ref{eq:alpha}) and $\mathbf{v}^{T}F^{(d)}\mathbf{u}$
(\ref{eq:bilinear form for G}), along with the identity $\sum_{d}w_{d}\boldsymbol{\Omega}_{d}=\mathbf{0}$,
\[
\mathbf{v}^{T}\frac{1}{\varepsilon}F_{0}\mathbf{u}=\frac{1}{\varepsilon}\sum_{d}w_{d}\mathbf{v}^{T}F^{(d)}\mathbf{u}=\sum_{\Gamma\in\mathcal{F}}\int_{\Gamma}\alpha\left\llbracket u\right\rrbracket \left\llbracket v\right\rrbracket dS.
\]
The remaining terms are slightly more technical, and Section~\ref{subsec:Face-matrix-terms}
proves that
\begin{equation}
\mathbf{v}^T\left(\mathbf{G}^T\cdot M_{t}^{-1}\mathbf{F}_{1}\right)\mathbf{u}=-\sum_{\Gamma\in\mathcal{F}}\int_{\Gamma}\frac{1}{3\sigma_{t}}\left\{ \mathbf{n}\cdot\nabla_{\mathbf{x}}v\right\} \left\llbracket u\right\rrbracket dS,\label{eq:diffusion coupling term 1}
\end{equation}
\begin{equation}
\mathbf{v}^T\left(\tilde{\mathbf{F}}_{1}\cdot M_{t}^{-1}\mathbf{G}\right)\mathbf{u}=\sum_{\Gamma\in\mathcal{F}}\int_{\Gamma}\frac{1}{3\sigma_{t}}\left\{ \mathbf{n}\cdot\nabla_{\mathbf{x}}u\right\} \left\llbracket v\right\rrbracket dS.\label{eq:diffusion coupling term 2}
\end{equation}
Together the above results combine to yield the identity in (\ref{eq:Deps and BMIP}).

\subsubsection{Face matrix terms in bilinear form\label{subsec:Face-matrix-terms}}

This section starts with a lemma expressing the action of $M_{t}^{-1}\mathbf{G}$
in the context of bilinear forms. 
\begin{lemma}
\label{lem:derivative}In the case of straight-edged meshes and 
constant opacities, $1/\sigma_{t}$, for each mesh element, $\kappa_{e}$,
$\left(M_{t}^{-1}\mathbf{G}\right)\mathbf{u}$ is related to $\frac{1}{\sigma_{t}}\nabla_{\mathbf{x}}u$
via
\begin{equation}
\left(M_{t}^{-1}\right)_{e,e}\mathbf{G}_{e,e}\left[\boldsymbol{u}_{e}\right]_{m}=\frac{1}{\sigma_{t}}\left[\left(\nabla_{\mathbf{x}}u_{e}\right)\left(\mathbf{x}_{e,m}\right)\right]_{m}.\label{eq:G and derivative}
\end{equation}
Moreover, for any bilinear form $\mathcal{B}\left(u,v\right)$ with
associated matrix $B$, $\mathcal{B}\left(u,v\right)=\mathbf{v}^TB\mathbf{u}$,
it holds that 
\[
\mathcal{B}\left(\frac{1}{\sigma_{t}}\partial_{x_{j}}u,v\right)=\mathbf{v}^T\left(B\left(M_{t}^{-1}G^{(j)}\right)\right)\mathbf{u}.
\]
\end{lemma}
\begin{proof}
Without loss of generality, expand $u\left(\mathbf{x}\right)$ in
a piecewise polynomial basis consisting of interpolating polynomials
$\left\{ u_{e,m}\left(\mathbf{x}\right)\right\} _{e,m}$,
\[
u\left(\mathbf{x}\right)=\sum_{m}u\left(\mathbf{x}_{e,m}\right)u_{e,m}\left(\mathbf{x}\right),\,\,\,\,\,\mathbf{x}\in\kappa_{e}.
\]
Since the mesh transformation from the reference element $\hat{\kappa}$
to the physical element $\kappa_{e}$ is linear, $\partial_{x_{j}}u_{e}\left(\mathbf{x}\right)$
is a polynomial of degree less than or equal to the degree of $u_{e}\left(\mathbf{x}\right)$,
and so
\[
\partial_{x_{j}}u_{e}\left(\mathbf{x}\right)=\sum_{n}u_{e}\left(\mathbf{x}_{e,n}\right)\partial_{x_{j}}u_{e,n}\left(\mathbf{x}\right)=\sum_{n}\left(\partial_{x_{j}}u_{e}\right)\left(\mathbf{x}_{e,n}\right)u_{e,n}\left(\mathbf{x}\right).
\]
Therefore, 
\[
\int_{\kappa_{e}}\left(\partial_{x_{j}}u_{e}\right)u_{e,m}d\mathbf{x}=\sum_{n}\partial_{x_{j}}u_{e}\left(\mathbf{x}_{e,n}\right)\int_{\kappa_{e}}u_{e,n}u_{e,m}d\mathbf{x}=\sum_{n}u_{e}\left(\mathbf{x}_{e,n}\right)\int_{\kappa_{e}}\left(\partial_{x_{j}}u_{e,n}\right)u_{e,m}d\mathbf{x}.
\]
Recalling that $G_{e}^{(j)}=\left[\int_{\kappa_{e}}u_{e,m}\left(\partial_{x_{j}}u_{e,n}\right)d\mathbf{x}\right]_{mn}$
and $M_{t,e}=\sigma_{t}\left[\int_{\kappa_{e}}u_{e,n}u_{e,m}d\mathbf{x}\right]_{mn}$,
we can write the above identity as 
\[
\sigma_{t}^{-1}\left(M_{t}\right)_{e,e}\left[\partial_{x_{j}}\varphi_{e}\left(\mathbf{x}_{m}\right)\right]_{m}=G_{e}^{(j)}\left[\varphi_{e}\left(\mathbf{x}_{m}\right)\right]_{m}.
\]
Applying $\left(M_{t}^{-1}\right)_{e,e}$ to both sides above yields
equation (\ref{eq:G and derivative}).

Now consider the bilinear form $\mathcal{B}\left(\sigma_{t}^{-1}\partial_{x_{j}}u,v\right)$
and suppose that the matrix $B$ is such that 
\[
\mathbf{v}^TB\mathbf{u}=\mathcal{B}\left(u,v\right),
\]
for any $u\left(\cdot\right)$ and $v\left(\cdot\right)$ in the DG
space. Then we use the following: if $\sigma_{t}$ is constant, then the bilinear form $\mathcal{B}\left(u,\partial_{x_{j}}v\right)$
corresponds to the matrix $B\left(M_{t}^{-1}G^{(j)}\right)$. Indeed,
letting $B_{e',e}$ denote the submatrix of $B$ corresponding to
elements $\kappa_{e'}$ and $\kappa_{e}$,
\begin{align*}
\mathcal{B}\left(\frac{1}{\sigma_{t}}\partial_{x_{j}}u,v\right) & =\sum_{e,e'}\sum_{m,n}\left[v_{e'}\left(\mathbf{x}_{e',m}\right)\right]_{m}\left(B_{e',e}\right)_{m,n}\frac{1}{\sigma_{t}}\left[\left(\partial_{x_{j}}u_{e}\right)\left(\mathbf{x}_{e,n}\right)\right]_{n}\\
 & =\sum_{e,e'}\mathbf{v}_{e'}^TB_{e',e}\left(\left(M_{t,e}^{-1}G_{e}^{(j)}\right)\mathbf{u}_{e}\right)\\
 & =\sum_{e,e'}\mathbf{v}_{e'}^T\left(B_{e',e}\left(M_{t,e}^{-1}G_{e}^{(j)}\right)\right)\mathbf{u}_{e}\\
 & =\mathbf{v}^T\left(B\left(M_{t}^{-1}G^{(j)}\right)\right)\mathbf{u}.
\end{align*}
Using equations (\ref{eq:F_def}) and (\ref{eq:tildeFd bilinear form}),
and the identities $\sum_{d}w_{d}\boldsymbol{\Omega}_{d}\boldsymbol{\Omega}_{d}^{T}=\frac{4\pi}{3}I$
and $\sum_{d}w_{d}\boldsymbol{\Omega}_{d}\left|\boldsymbol{\Omega}_{d}\cdot\mathbf{n}\right|=\mathbf{0}$,
\begin{equation}
\begin{split}
\label{eq:F1 inner product}
\mathbf{v}^T\left(\mathbf{F}_{1}\right)_{j}\mathbf{u} & =-\sum_{\Gamma\in\mathcal{F}}\int_{\Gamma}\left(\frac{1}{4\pi}\sum_{d}w_{d}\left(\boldsymbol{\Omega}_{d}\right)_{j}\boldsymbol{\Omega}_{d}^T\right)\mathbf{n}\left\llbracket u\right\rrbracket \left\{ v\right\} dS+  \\
 & \qquad\quad\frac{1}{2}\sum_{\Gamma\in\mathcal{F}^{\text{i}}}\int_{\Gamma}\left(\frac{1}{4\pi}\sum_{d}w_{d}\left(\boldsymbol{\Omega}_{d}\right)_{j}\left|\boldsymbol{\Omega}_{d}\cdot\mathbf{n}\right|\right)\left\llbracket u\right\rrbracket \left\llbracket v\right\rrbracket dS \\
 & =-\frac{1}{3}\sum_{\Gamma\in\mathcal{F}}\int_{\Gamma}\mathbf{n}_{j}\left\llbracket u\right\rrbracket \left\{ v\right\} dS ,
\end{split}
\end{equation}
where $\mathbf{n}_{j}$ denotes the $j$th component of the normal
vector $\mathbf{n}$. Similarly,
\begin{equation}
\begin{split}
\label{eq:tildeF1 inner product}
\mathbf{v}^T\left(\tilde{\mathbf{F}}_{1}\right)_{j}\mathbf{u} & =\sum_{\Gamma\in\mathcal{F}}\int_{\Gamma}\left(\frac{1}{4\pi}\sum_{d}w_{d}\left(\boldsymbol{\Omega}_{d}\right)_{j}\boldsymbol{\Omega}_{d}^T\right)\mathbf{n}\left\{ u\right\} \left\llbracket v\right\rrbracket dS +  \\
 & \qquad\quad\frac{1}{2}\sum_{\Gamma\in\mathcal{F}}\int_{\Gamma}\left(\frac{1}{4\pi}\sum_{d}w_{d}\left(\boldsymbol{\Omega}_{d}\right)_{j}\left|\boldsymbol{\Omega}_{d}\cdot\mathbf{n}\right|\right)\left\llbracket u\right\rrbracket \left\llbracket v\right\rrbracket dS  \\
 & =\frac{1}{3}\sum_{\Gamma\in\mathcal{F}}\int_{\Gamma}\mathbf{n}_{j}\left\{ u\right\} \left\llbracket v\right\rrbracket dS.
\end{split}
\end{equation}
Applying Lemma~\ref{lem:derivative} yields equations (\ref{eq:diffusion coupling term 1})
and (\ref{eq:diffusion coupling term 2}).
\end{proof}

\subsection{Fixed-point iteration on HO meshes\label{subsec:Mesh-cycles}}

Theorem ~\ref{th:mesh_cycles} follows from the following Lemma.

\begin{lemma}
\label{prop:mesh_cycles}
Consider a linear system of the form in \eqref{eq:bigsys}, with condensed notation
\begin{equation}
\label{eq:gensys}
(I + \varepsilon H - B)\boldsymbol{\psi}^{(d)} = \mathbf{s}^{(d)}.
\end{equation}
Denote $\mathcal{H} := I + \varepsilon H - B$, and consider a matrix splitting $H = H_{\leq} + H_{>}$. Define
$\widetilde{M}^{-1} = (I + \varepsilon H)^{-1}$ as the preconditioner
associated with inverting $I + \varepsilon H$. Now, fix $B\boldsymbol{\psi}^{(d)}$ and move it to the right-hand side, for the modified linear system
\begin{equation}
\label{eq:fixsys}
(I + \varepsilon H)\boldsymbol{\psi}^{(d)} = \mathbf{s}^{(d)} + B\boldsymbol{\psi}_0^{(d)}.
\end{equation}
Define $\widehat{M}_k^{-1}$ as the preconditioner associated with performing $k$ fixed-point iterations on \eqref{eq:fixsys},
with approximate inverse $\widehat{M}_1^{-1} = (I + \varepsilon H_{\leq})^{-1}$. Then, applying $\widetilde{M}^{-1}$ and
$\widehat{M}_k^{-1}$ as preconditioners for \eqref{eq:gensys} is related via
\[
\widehat{M}_k^{-1}\mathcal{H} = \left(I- \left(-\varepsilon (I +\varepsilon H_{\leq})^{-1}H_{>}\right)^k\right) \widetilde{M}^{-1}\mathcal{H}.
\]
\end{lemma}

\begin{proof}[Proof of Theorem \ref{th:mesh_cycles}]

Consider a problem of the form 
\[
(I + \varepsilon H_{\leq} + \varepsilon H_{>} - B)\boldsymbol{\psi}^{(d)} = \mathbf{s}^{(d)},
\]
where $\mathcal{H} := (I + \varepsilon H_{\leq} + \varepsilon H_{>} - B)$. Note the following identities, which will be used regularly:
\begin{align*}
(I + \varepsilon H_{\leq} + \varepsilon H_{>} - B)^{-1} & = 
  \left[I - (I +\varepsilon H_{\leq} + \varepsilon H_{>})^{-1}B\right]^{-1}(I +\varepsilon H_{\leq} + \varepsilon H_{>})^{-1}, \\
(I + \varepsilon H_{\leq} + \varepsilon H_{>} - B)^{-1} & =
  \left[I - (I +\varepsilon H_{\leq})^{-1}(-\varepsilon H_{>}+B)\right]^{-1}(I +\varepsilon H_{\leq})^{-1} , \\
(I + \varepsilon H_{\leq} + \varepsilon H_{>})^{-1} & = \left(I + \varepsilon(I+\varepsilon H_{\leq})^{-1}H_{>}\right)^{-1}(I + \varepsilon H_{\leq})^{-1}.
\end{align*}
First, consider a single fixed-point iteration, where we invert $I + \varepsilon H_{\leq} + \varepsilon H_{>}$. Define
$\widetilde{M}^{-1} = (I + \varepsilon H_{\leq} + \varepsilon H_{>})^{-1}$. Then, the preconditioned linear system 
is given by $\widetilde{M}^{-1}(\mathcal{H}\boldsymbol{\psi}^{(d)}-\mathbf{s}^{(d)} ) = \mathbf{0}$, where
\begin{align*}
\widetilde{M}^{-1}\mathcal{H} & =  (I + \varepsilon H_{\leq} + \varepsilon H_{>})^{-1}(I + \varepsilon H_{\leq} + \varepsilon H_{>} - B) \\
  & = I - (I +\varepsilon H_{\leq} + \varepsilon H_{>})^{-1}B.
\end{align*}

Now suppose we only invert $I + \varepsilon H_{\leq}$, that is, our preconditioner is given by $\widehat{M}_1^{-1}=(I+\varepsilon H_{\leq})^{-1}$.
This arises, for example, in the case of cycles in the mesh, where we can only directly invert the block lower triangular part. 
In the interest of asymptotics, additionally consider moving $B\boldsymbol{\psi}$ to the right-hand side and applying multiple iterations of
$\widehat{M}_1^{-1}$ to the modified linear system,
$\widehat{\mathcal{H}}\boldsymbol{\psi}^{(d)} = \widehat{\mathbf{s}}^{(d)}$, given by
\[
(I + \varepsilon H_{\leq} + \varepsilon H_{>})\boldsymbol{\psi}^{(d)} =
  \mathbf{s}^{(d)} +  B{\boldsymbol{\psi}}_0^{(d)},
\]
where ${\boldsymbol{\psi}}_0^{(d)}$ is fixed for all iterations.
In a fixed-point sense, this is equivalent to
\[
\boldsymbol{\psi}_{k+1}^{(d)} =
  \boldsymbol{\psi}_k^{(d)} + 
  (I + \varepsilon H_{\leq})^{-1}(\mathbf{s}^{(d)} + B\boldsymbol{\psi}_0^{(d)} -
  \varepsilon(\mathcal{H}_{\leq} + \mathcal{H}_{>}) \boldsymbol{\psi}_k^{(d)}),
\]
with error propagation given by
\begin{align*}
I - \widehat{M}_1^{-1}\widehat{\mathcal{H}} & = I - (I + \varepsilon H_{\leq})^{-1}\widehat{\mathcal{H}} \\
& = -\varepsilon(I + \varepsilon H_{\leq})^{-1}H_{>}.
\end{align*}
Then, we are interested in the preconditioner $\widehat{M}_k$ that results
from taking powers of $I - \widehat{M}_k^{-1}\widehat{\mathcal{H}} = (I - \widehat{M}_1^{-1}\widehat{\mathcal{H}})^k $.
Solving for $\widehat{M}_k^{-1}$,
\begin{align*}
\widehat{M}_k^{-1}\widehat{\mathcal{H}} & = I - \left(-\varepsilon(I + \varepsilon H_{\leq})^{-1}H_{>}\right)^k , \\
\widehat{M}_k^{-1}& = \widehat{\mathcal{H}}^{-1} -  \left(-\varepsilon(I + \varepsilon H_{\leq})^{-1}H_{>}\right)^k\widehat{\mathcal{H}}^{-1} \\
& = \left(I - -\varepsilon(I + \varepsilon H_{\leq})^{-1}H_{>}\right]){-1}(I +\varepsilon H_{\leq})^{-1} - \\&\hspace{8ex}
  \left(-\varepsilon(I + \varepsilon H_{\leq})^{-1}H_{>}\right)^k\left(I - (I +\varepsilon H_{\leq})^{-1}(-\varepsilon H_{>}+B)\right)^{-1}(I +\varepsilon H_{\leq})^{-1} \\
& = \left[ \sum_{\ell=0}^\infty \left( -\varepsilon(I + \varepsilon H_{\leq})^{-1}H_{>}\right)^\ell - 
  \sum_{\ell=k}^\infty \left( -\varepsilon(I + \varepsilon H_{\leq})^{-1}H_{>}\right)^\ell \right] (I +\varepsilon H_{\leq})^{-1} \\
& = \left[\sum_{\ell=0}^{k-1} \left( -\varepsilon(I + \varepsilon H_{\leq})^{-1}H_{>}\right)^\ell \right] (I +\varepsilon H_{\leq})^{-1}.
\end{align*}

Now, suppose we apply $\widehat{M}_k^{-1}$ as a preconditioner for the original
linear system, $\mathcal{H}\boldsymbol{\psi}^{(d)} = \mathbf{s}^{(d)}$, and
consider the difference between $\widehat{M}_k^{-1}$ and $\widetilde{M}^{-1}$:
\begin{align*}
\widehat{M}_k^{-1}  - \widetilde{M}^{-1} & = \left[\sum_{\ell=0}^{k-1} \left( -\varepsilon
  (I +\varepsilon H_{\leq})^{-1}H_{>}\right)^\ell \right] (I +\varepsilon H_{\leq})^{-1} -
  (I + \varepsilon H_{\leq} + \varepsilon H_{>})^{-1} \\
& = \left[\sum_{\ell=0}^{k-1} \left( -\varepsilon
  (I +\varepsilon H_{\leq})^{-1}H_{>}\right)^\ell \right] (I +\varepsilon H_{\leq})^{-1} -
  \left(I + \varepsilon(I+\varepsilon H_{\leq})^{-1}H_{>}\right)^{-1}(I + \varepsilon H_{\leq})^{-1} \\
& = \left[\sum_{\ell=0}^{k-1} \left( -\varepsilon
  (I +\varepsilon H_{\leq})^{-1}H_{>}\right)^\ell - \sum_{\ell=0}^{\infty} \left( -\varepsilon
  (I +\varepsilon H_{\leq})^{-1}H_{>}\right)^\ell\right] (I +\varepsilon H_{\leq})^{-1} \\
& = -\left[ \sum_{\ell=k}^{\infty} \left( -\varepsilon
  (I +\varepsilon H_{\leq})^{-1}H_{>}\right)^\ell \right] (I +\varepsilon H_{\leq})^{-1} \\
& = -\left(-\varepsilon(I +\varepsilon H_{\leq})^{-1}H_{>}\right)^k \left[ \sum_{\ell=0}^{\infty} \left( -\varepsilon
  (I +\varepsilon H_{\leq})^{-1}H_{>}\right)^\ell \right] (I +\varepsilon H_{\leq})^{-1} \\
& = -\left(-\varepsilon(I +\varepsilon H_{\leq})^{-1}H_{>}\right)^k \widetilde{M}^{-1}.
\end{align*}

Then,
\[
\widehat{M}_k^{-1}\mathcal{H} = \left[ I - \left(-\varepsilon (I +\varepsilon H_{\leq})^{-1}H_{>}\right)^k\right]\widetilde{M}^{-1}\mathcal{H}.
\]
\end{proof}

\section{Numerical experiments\label{sec:numerical}}
In this section we report numerical results
\cblue{from the three major approaches presented in the previous sections,
namely, the SIP DSA (Theorem \ref{th:MIP}),
its IP modification (Section \ref{sec:consistent discretization}),
and the additive DSA preconditioner (Theorem \ref{th:new_DSA}).} \cred{We also show that performing two 
additional transport sweeps in between DSA steps  (Theorem~\ref{th:mesh_cycles}) can greatly accelerate (or prevent divergence of) source iteration on 
HO meshes with mesh cycles.}

We present calculations and comparisons on highly curved 2D and 3D meshes that
are obtained from moving mesh hydrodynamic simulations \cite{Dobrev2012}.
The methods in this paper are implemented by utilizing the finite element
infrastructure provided by the MFEM finite element library \cite{mfem}.
\subsection{DSA preconditioning on a HO Lagrangian mesh}
\label{sec_numerical_2D}

This section uses DSA to solve the discrete transport equations (\ref{eq:Tmat1})
on a HO hydrodynamics mesh generated from a purely Lagrangian simulation
of the ``triple point'' problem \cite{triplePoint}, which is displayed
in Figure~\ref{fig:triple-point mesh} \cred{(the spatial domain for this problem is $[0,7] \times [0,3]$)}. The mesh is $3^{\text{rd}}$-order
mesh, that is, cubic polynomials are used to map the reference element
to physical elements, and our DG discretization uses $3^{\text{rd}}$-order
local basis functions. \cred{We also use an S2 quadrature discretization.}

\begin{figure}[h!]
  \centerline
  {
    \includegraphics[width=0.60\textwidth]{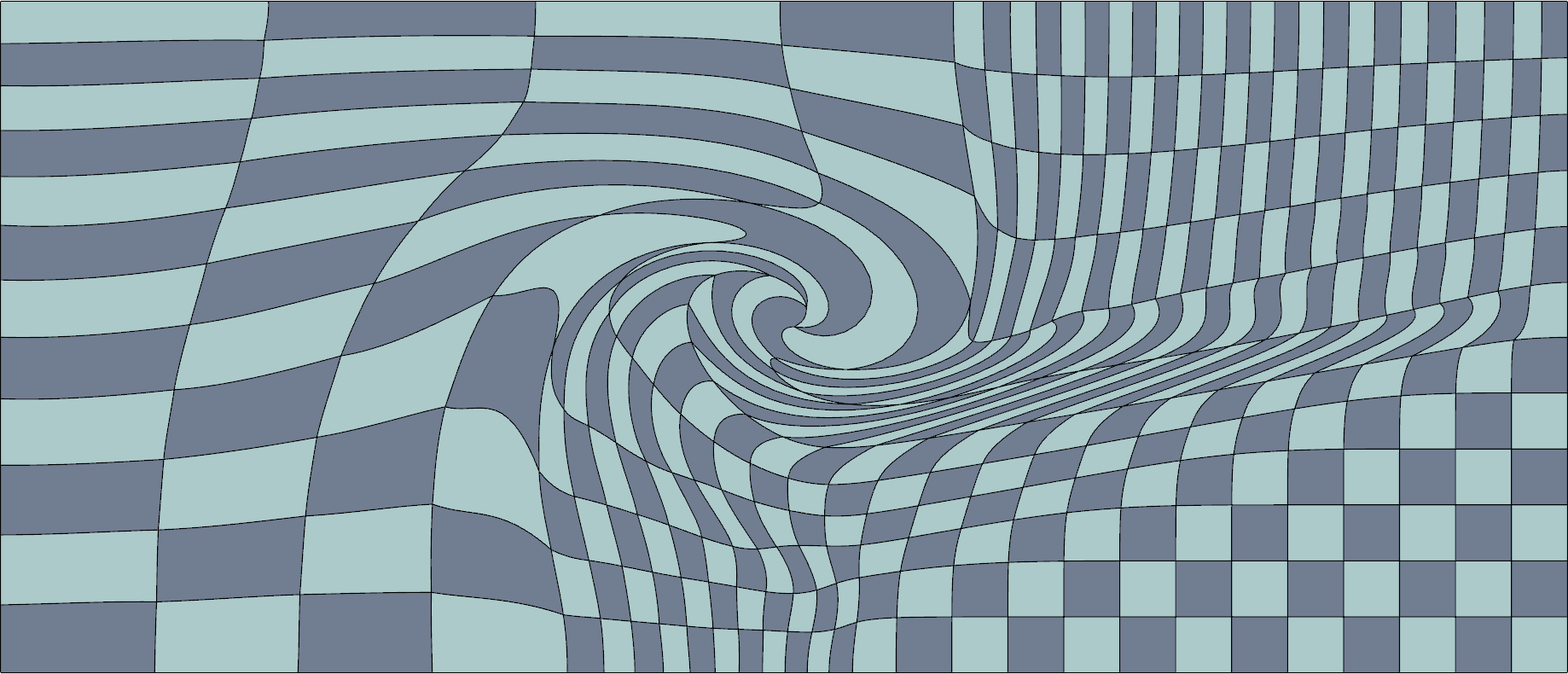}
  }
  \caption
  {
     A ``triple point'' $3^{\text{rd}}$-order Lagrangian mesh.
  }
  \label{fig:triple-point mesh}
\end{figure}

\begin{figure}
\centering
\label{fig:Iteration-triple-point}
\subfloat[$\varepsilon = 10^{-4}$]{ \includegraphics[width = 2in]{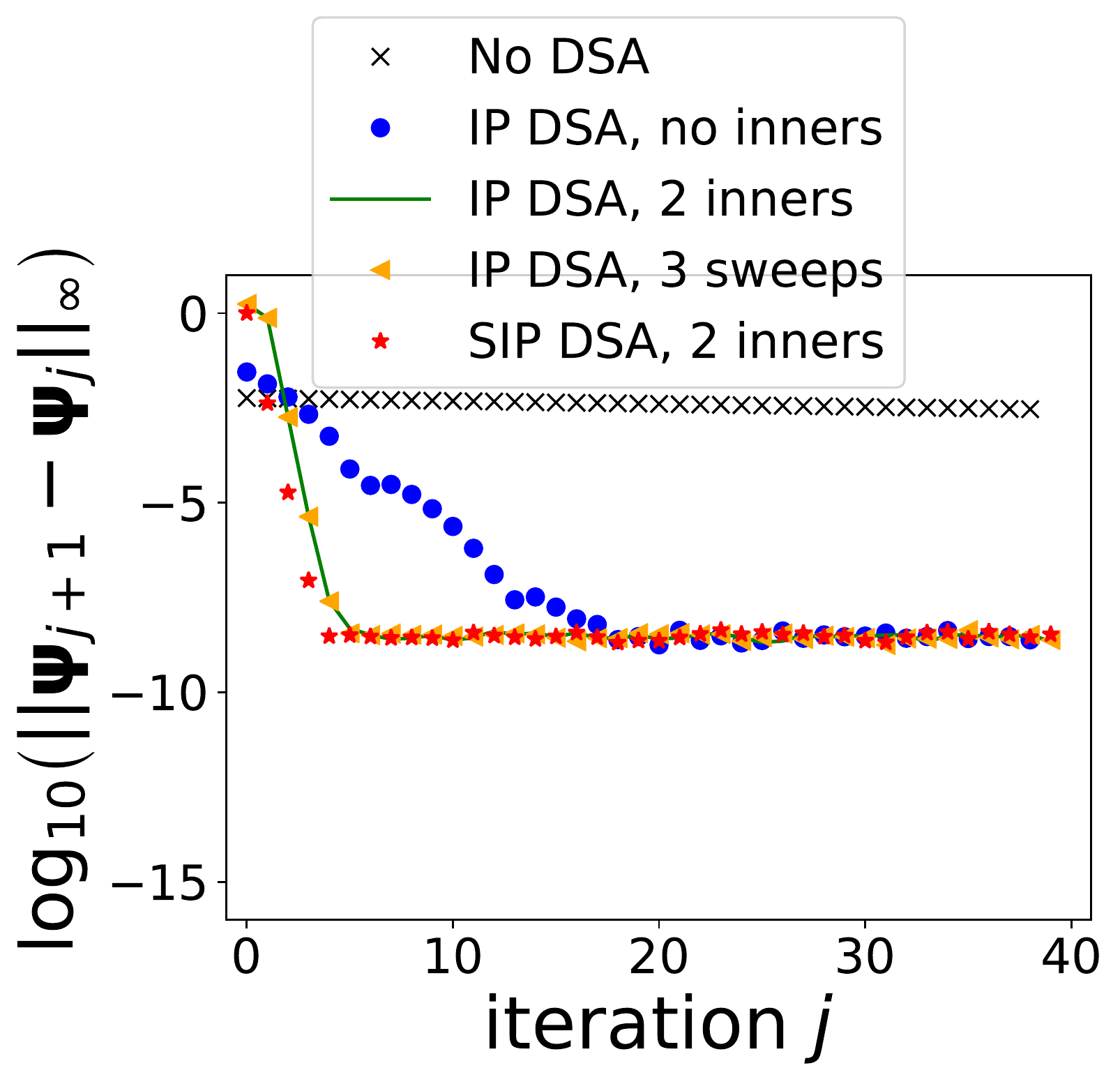} }
\subfloat[$\varepsilon = 10^{-3}$]{ \includegraphics[width = 2in]{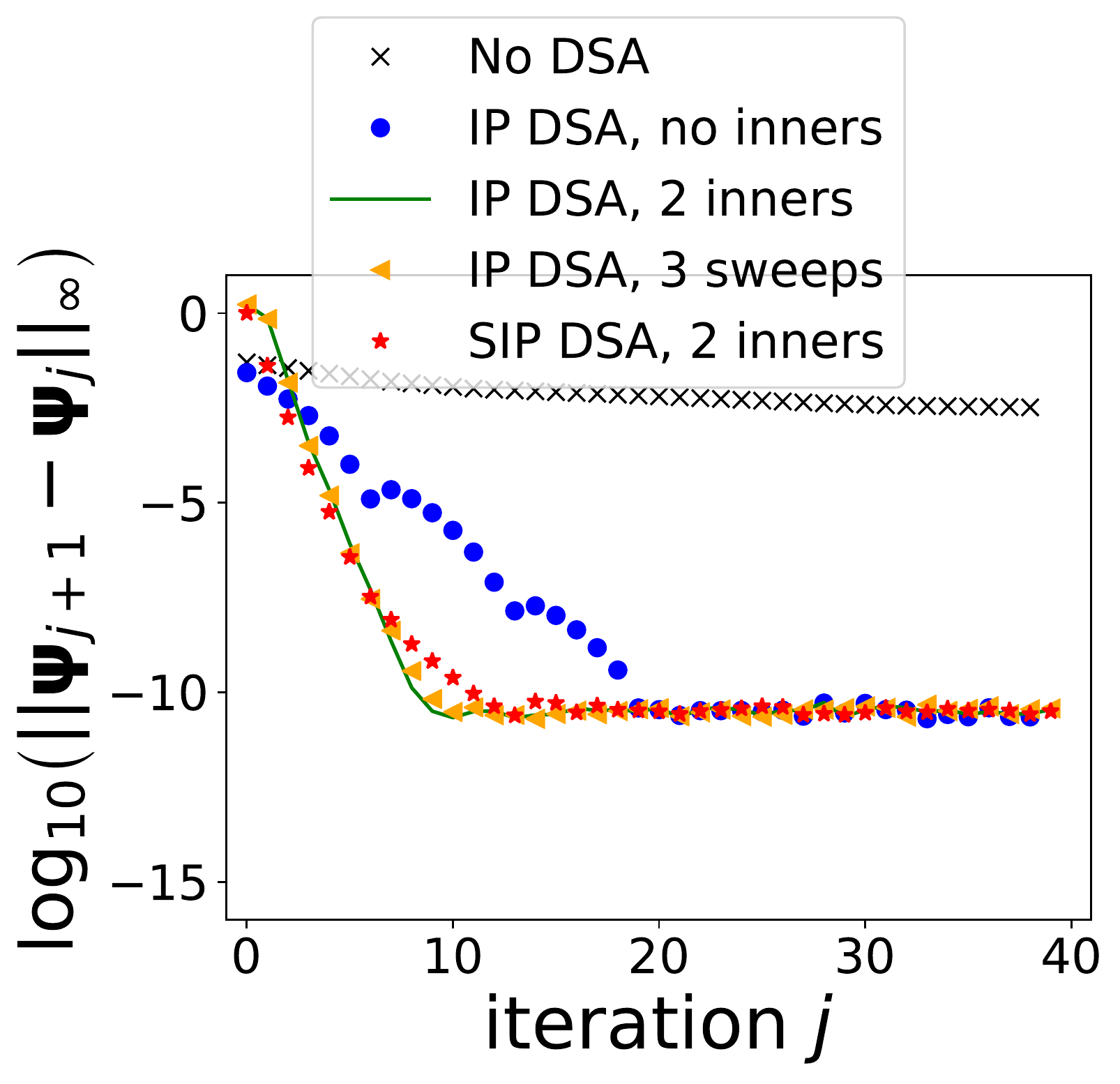}  }
\subfloat[$\varepsilon = 10^{-2}$]{ \includegraphics[width = 2in]{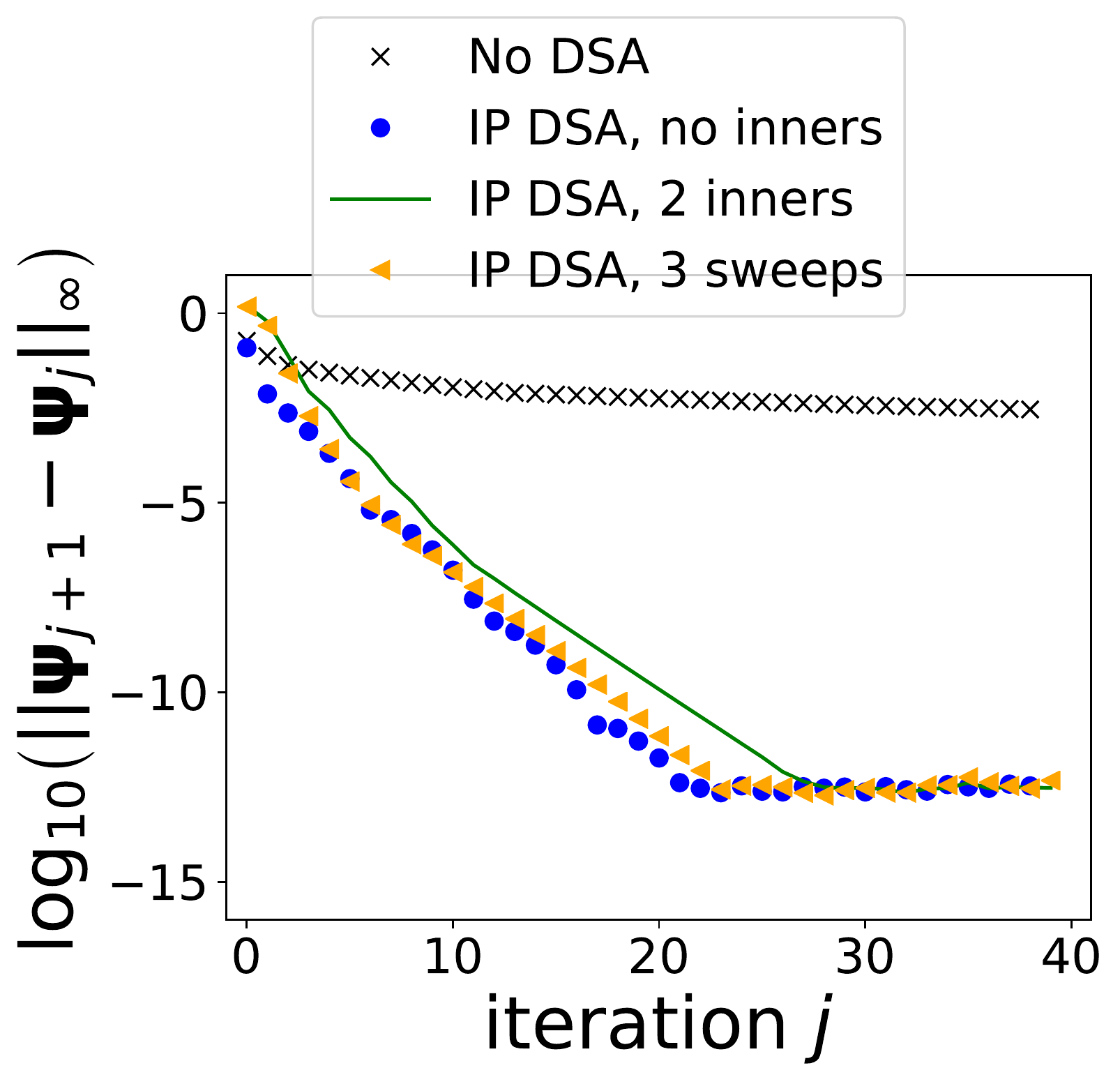} }\\
\subfloat[$\varepsilon = 10^{-1}$]{ \includegraphics[width = 2in]{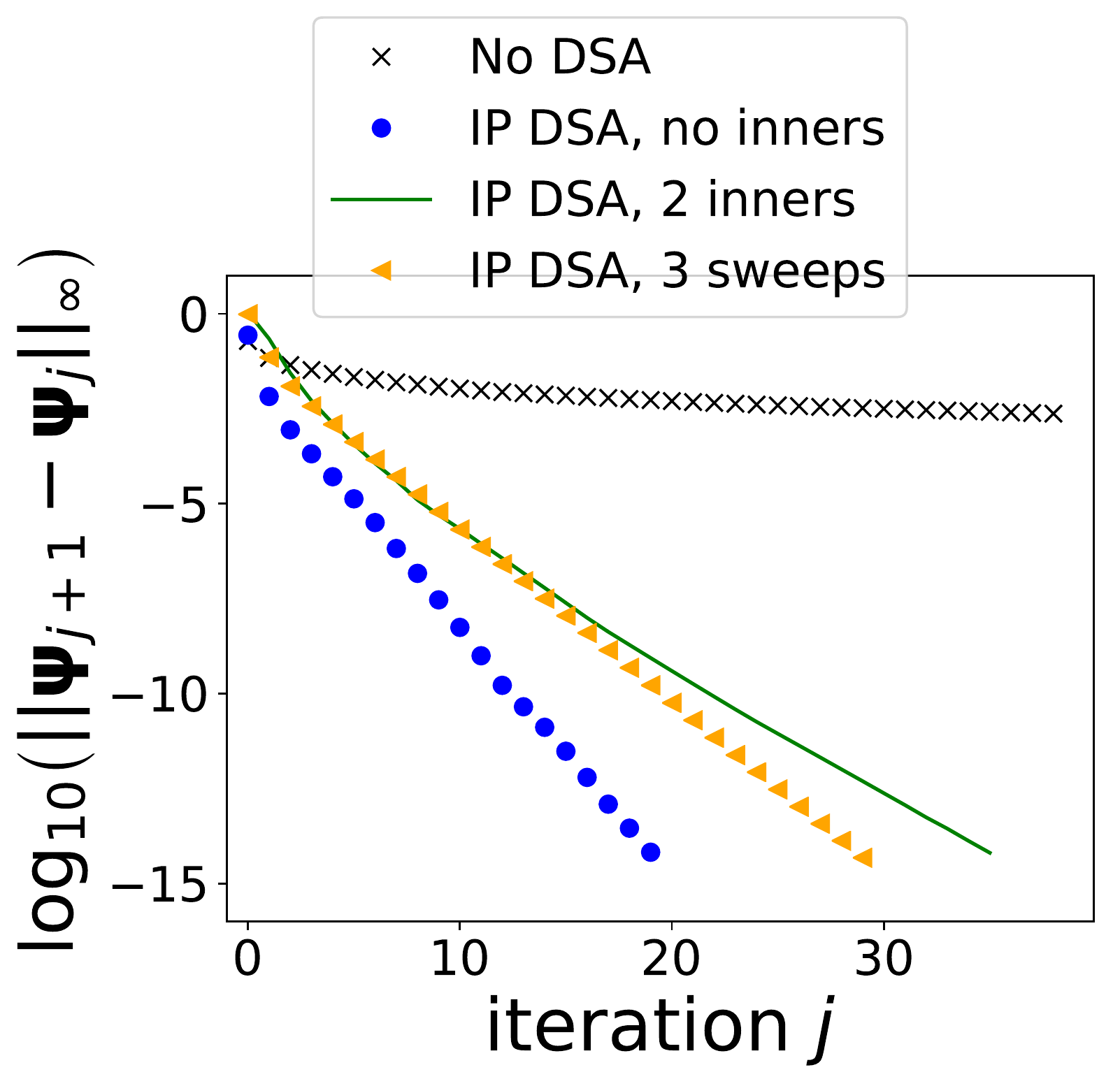} }
\subfloat[$\varepsilon = 0.75$]{ \includegraphics[width = 2in]{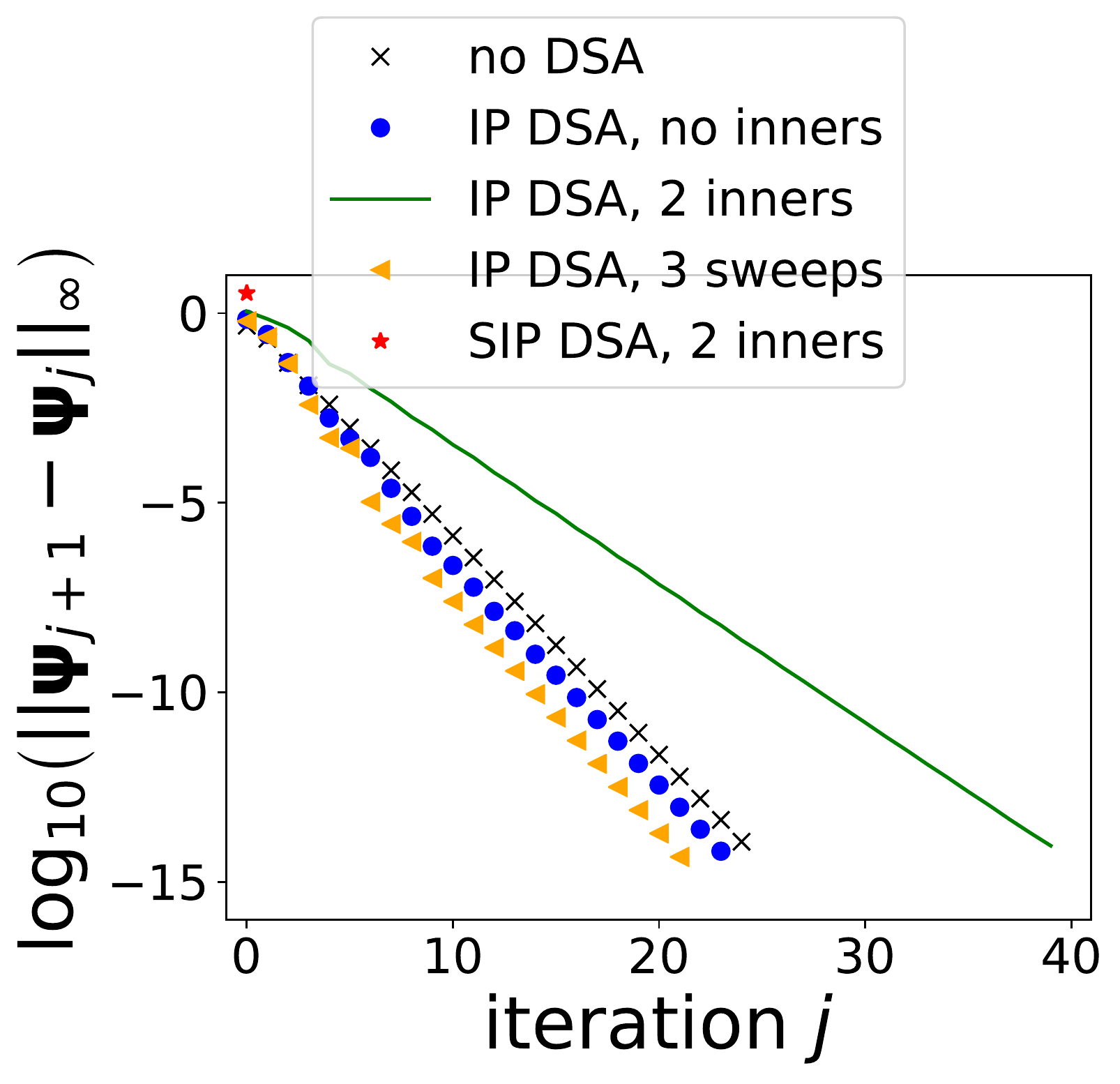}  }
\caption{Iteration error estimate $\|\boldsymbol{\psi}_{j+1}-\boldsymbol{\psi}_{j+1}\|_{\infty}$
as a function of iteration index $j$ on the triple point problem,
with and without DSA preconditioning using the DSA matrix \eqref{eq:modified DSA matrix}. \cgreen{In all cases considered, we always perform $40$ ``iterations".} Although the iteration index $j$ has a
different interpretation for the five displayed cases ``no DSA", ``IP DSA, no inners", ``IP DSA, 2 inners", ``IP DSA, 3 sweeps", and ``SIP DSA, 2 inners", each iteration involves the same
number of transport sweeps. For example, ``IP DSA, 2 inners" refers to the nonsymmetric interior penalty version of DSA with three total transport sweeps between DSA steps (but with a fixed scalar flux), and ``IP DSA, 3 sweeps" refers to the nonsymmetric interior penalty version of DSA with a DSA step three transport sweeps (where the scalar flux is updated after each sweep). \cgreen{Similarly, each iteration index in the ``no DSA" option corresponds to performing $3$ transport sweeps per iteration index.} In the plots where ``SIP DSA, 2 inners" is not displayed, the fixed-point iteration diverged for this case.}
\end{figure}

\noindent
For this problem, we use constant opacities
\[
\sigma_{t}\left(x_{1},x_{2}\right)=\frac{1}{\varepsilon},\,\,\,\,\,\sigma_{a}=\varepsilon,
\]
and a smooth (but arbitrary) source term
\[
q\left(x_{1},x_{2}\right)=\varepsilon\cos^{2}\left(2x_{1}+x_{2}\right),
\]
where $\varepsilon$ is the characteristic mean free path. In our numerical experiments, we vary $\varepsilon$ from relatively optically thin regimes $\varepsilon = .75$, to increasingly optically thick regimes  $\varepsilon=10^{-j}$, for $j=1,2,3,4$.
Last, constant inflow boundary conditions are applied,
\[
\psi_{d}\left(\mathbf{x}\right)=1\hspace{3ex}\textnormal{when}\hspace{3ex}
  \boldsymbol{\Omega}_{d}\cdot\mathbf{n}\left(\mathbf{x}\right)<0 \,\,\,\,\, \text{and} \,\,\,\,\, \mathbf{x}\in \partial \mathcal{D} .
\]

Figure~\ref{fig:Iteration-triple-point} shows the iteration error estimate $\|\boldsymbol{\psi}_{j+1}-\boldsymbol{\psi}_{j+1}\|_{\infty}$
as a function of iteration index $j$, with and without DSA preconditioning. Recall, due to cycles in the mesh, the transport equation
for a fixed angle cannot be easily inverted, so we invert the block lower triangular part of the matrix, and refer to this as a
``transport sweep.'' When DSA preconditioning is included, we consider using a single transport sweep with lagging
between DSA steps, as well as using two ``inner sweeps,'' where the scalar flux is not updated, followed by one normal sweep
with lagging between DSA steps (see Theorem \ref{th:mesh_cycles}). Finally, we also consider performing, between every
DSA step, three transport sweeps, where the scalar flux is updated after each sweep.
To ensure a fair comparison, the iteration index $j$ in all five cases displayed in Figure~\ref{fig:Iteration-triple-point}
accounts for \textit{the same number of
transport sweeps}; however, because of this, each case has a different interpretation:

\begin{enumerate}
\item In the ``no DSA" case, the iteration index $j$ corresponds to three
  applications of the fixed-point iteration without any DSA, i.e.,
  $(\text{sweep, update flux})^3$;
  for example, $j=10$ corresponds to 30 fixed-point iterations.
\item In the ``IP DSA, no inners" case, the iteration index $j$ represents
  three applications of a transport sweep and nonsymmetric interior penalty
  (IP) DSA step, i.e., $(\text{sweep, update flux, IP DSA})^3$.
\item In the ``IP DSA, $2$ inners" case, $j$ corresponds
  to three applications of a transport sweep followed by a scalar flux update
  and a single IP DSA step, that is, ((sweep$)^3$, update flux, IP DSA).
\item In the ``IP DSA, 3 sweeps" case, the index $j$ represents three
  applications of both a transport sweep and scalar flux update,
  followed by a nonsymmetric interior penalty (IP) DSA step, i.e.,
  ($(\text{sweep, update flux})^3$, IP DSA).
\item Finally, the ``SIP DSA, $2$ inners" case is the same as the
  ``IP DSA, $2$ inners" case, but with the symmetric interior penalty DSA
  matrix used instead.
\end{enumerate}

Because the sweep is typically computationally 
much more expensive than the DSA step, each iteration index in Figure~\ref{fig:Iteration-triple-point} approximately represents the
same computational work for each case. In particular, for small $\varepsilon$, Figure~\ref{fig:Iteration-triple-point} provides numerical
confirmation of the asymptotic result Theorem \ref{th:mesh_cycles}. Using three sweeps before each IP DSA step leads to a $4\times $
speedup for $\varepsilon = 10^{-4}$ when using the nonsymmetric IP DSA matrix (at a slightly lesser cost as well, due to two less diffusion solves), although the cost increases
in the optically thin regime relative to using no additional transport sweeps. In addition, although we didn't show this in Figure~\ref{fig:Iteration-triple-point}, the SIP DSA variant actually diverges for $\varepsilon = 10^{-3}$ and $\varepsilon = 10^{-4}$ when the inner iterations are not performed.

Table~\ref{tab:The-final-residual} displays the $L^{\infty}$ residuals of the final iterates,
\begin{equation}
\max_{d}\left\Vert \left(\boldsymbol{\Omega}_{d}\cdot\mathbf{G}+F^{(d)}+\frac{1}{\varepsilon}M_{t}\right)\boldsymbol{\psi}^{(d)}-\frac{1}{4\pi}\left(\frac{1}{\varepsilon}M_{t}-\varepsilon M_{a}\right)\boldsymbol{\varphi}-\frac{1}{4\pi}\left(\boldsymbol{q}_{\text{inc}}^{(d)}+\varepsilon\boldsymbol{q}^{(d)}\right)\right\Vert _{\infty},\label{eq:residual}
\end{equation}
as well as the iterations counts. Together, Figure~\ref{fig:Iteration-triple-point}  and Table~\ref{tab:The-final-residual}
confirm that DSA preconditioning on the HO mesh is effective across a wide range of characteristic mean free
paths. Interestingly, although using three transport sweeps between DSA steps is more effective for small $\varepsilon$,
for larger values of $\varepsilon$ it is best to apply a DSA step after each sweep. 

\begin{table}[h!]
\begin{center}
\caption{The final residual (\ref{eq:residual}) after 40 iterations with and without DSA preconditioning.\label{tab:The-final-residual}}
\begin{tabular}{|c|c|c|c|c|cl}
\hline 
$\varepsilon$ & IP DSA, 3 sweeps & IP DSA, 2 inners & SIP DSA, 2 inners & IP DSA, no inners & no DSA   \tabularnewline
\hline 
\hline 
 0.75 & 4.84e-15  &   2.37e-15 & 2.31e-15  & 3.02e-15 & 3.23e-15  \tabularnewline
\hline 
  1e-1 & 1.50e-14  & 3.34e-15   & diverged &  4.24e-15 & 6.92e-03   \tabularnewline
\hline 
 1e-2 & 1.95e-13   & 3.04e-13 & diverged & 2.05e-13 &  6.16e-02  \tabularnewline
\hline 
 1e-3 &  1.67e-11  &  2.04e-11  & 1.93e-11 & 2.23e-11 & 1.41e-01 \tabularnewline
\hline 
  1e-4 &  1.97e-09  & 1.69e-09   & 1.98e-09 & 2.77e-09 & 8.56e-01  \tabularnewline
\hline 
\end{tabular}
\end{center}
\end{table}

Note that as $\varepsilon$ gets smaller than $10^{-4}$, the DSA
preconditioner begins to degrade in effectiveness and ultimately leads
to a divergent fixed-point iteration. This degradation in efficiency
is likely due to the fact that the condition number of the system
(\ref{eq:psi_d}) scales like $\mathcal{O}\left(\varepsilon^{-2}\right)$,
and for smaller values of $\varepsilon$ the delicate cancellations
in the derivation of the DSA preconditioner can no longer be adequately
captured in floating point arithmetic.

\subsection{DSA preconditioning on a 3D curved mesh}
\label{sec_numerical_3D}

This section uses DSA to solve the discrete transport equations (\ref{eq:Tmat1})
on a HO hydrodynamics mesh generated from a purely Lagrangian simulation
of a 3D Rayleigh-Taylor instability (see Figure \ref{fig_3D_mesh}).
Again we utilize a $3^{\text{rd}}$-order mesh and
$3^{\text{rd}}$-order basis functions. \cred{ We also use an S4 quadrature discretization. }

\begin{figure}[!ht]
\label{fig_3D_mesh}
\centerline
{ \includegraphics[width=0.6\textwidth]{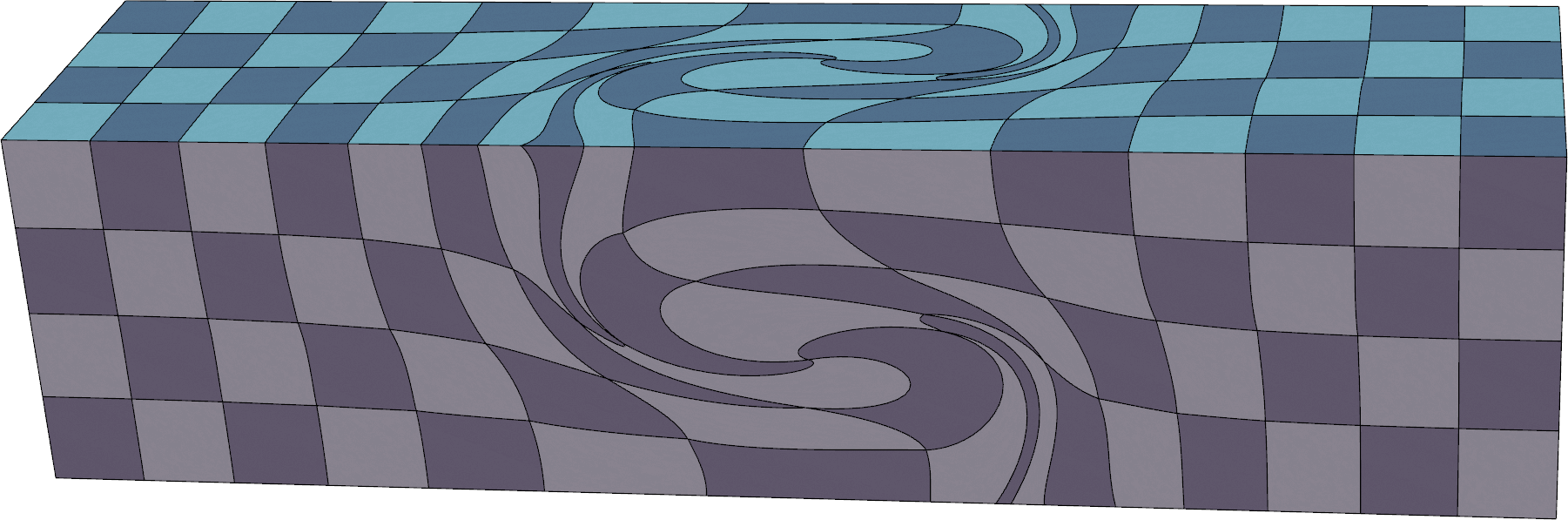} }
\vspace{4mm}
\centerline
{ \includegraphics[width=0.6\textwidth]{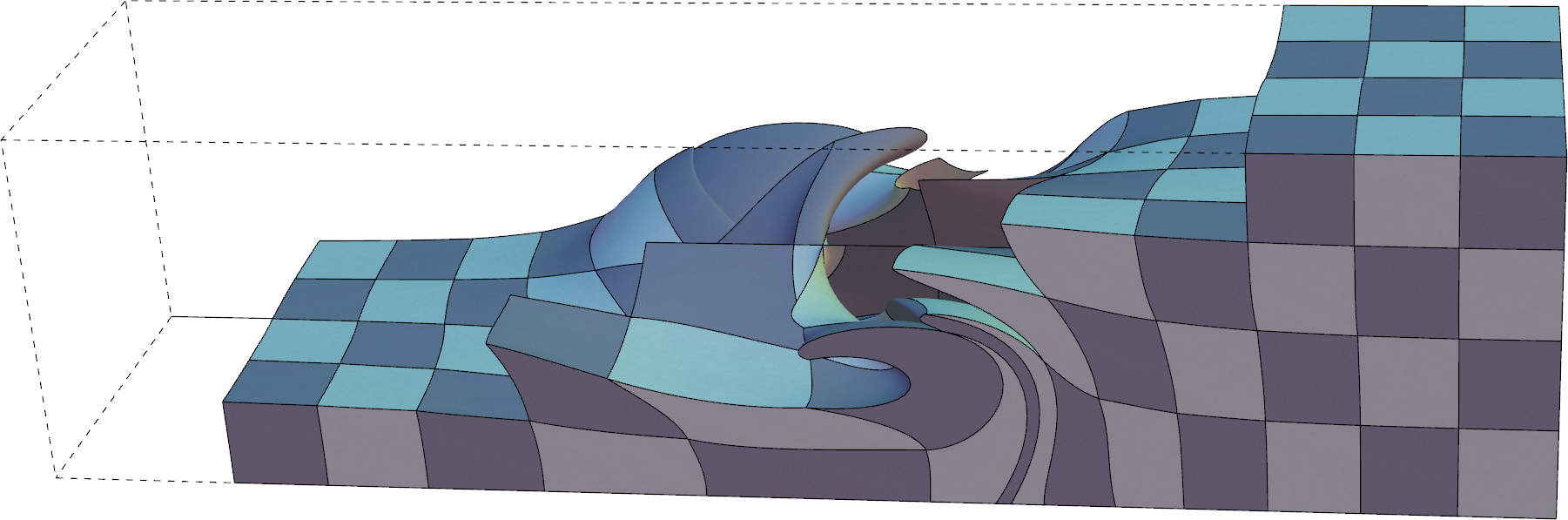} }
\caption { Cubic 3D mesh (top) resulting from a Lagrangian simulation of the
           Rayleigh-Taylor instability.
           A subset of the mesh elements is shown in the bottom. }
\end{figure}

\noindent
For this problem we use spatially dependent opacities
\[
\sigma_{t}(x_1, x_2, x_3) = \frac{x_1^2 + x_1 x_2 + 1}{\varepsilon},
\qquad
\sigma_{a}(x_1, x_2, x_3) = \frac{x_1^2 + x_1 x_2 + 1}{\varepsilon} -
                            \varepsilon(x_1^2 + x_1 x_2 + 0.5),
\]
and a smooth source term
\[
q(x_1, x_2, x_3) = \sin^{2}( 4 x_1 + 2 x_2 + 2 x_2 x_3 ) + 1,
\]
where $\varepsilon$ is the characteristic mean free path.
As in the previous section, we vary $\varepsilon$ from relatively optically
thin regimes $\varepsilon = 0.75$, to increasingly optically thick regimes
$\varepsilon=10^{-j}$, for $j=1,2,3,4$.
Lastly, constant inflow boundary conditions are applied:
\[
\psi_d(\mathbf{x}) = 1 \quad \textnormal{when} \quad
  \boldsymbol{\Omega}_{d} \cdot \mathbf{n}(\mathbf{x}) < 0
  \quad \text{and} \quad \mathbf{x}\in \partial \mathcal{D}.
\]

We repeat the numerical experiments from Section \ref{sec_numerical_2D}
to this 3D problem, using the above configuration.
The notation in Figure \ref{fig_iteration_3D} and Table \ref{tab_residual_3D}
follows the notation in Section \ref{sec_numerical_2D}.
We observe that all DSA preconditioning options for this 3D problem lead to
similar convergence trends as in the 2D problem.
{\color{black} Again, best convergence is achieved by the IP DSA options. }

\begin{figure}[!ht]
\label{fig_iteration_3D}
\centering
\subfloat[$\varepsilon = 10^{-4}$]
{ \includegraphics[width = 2in]{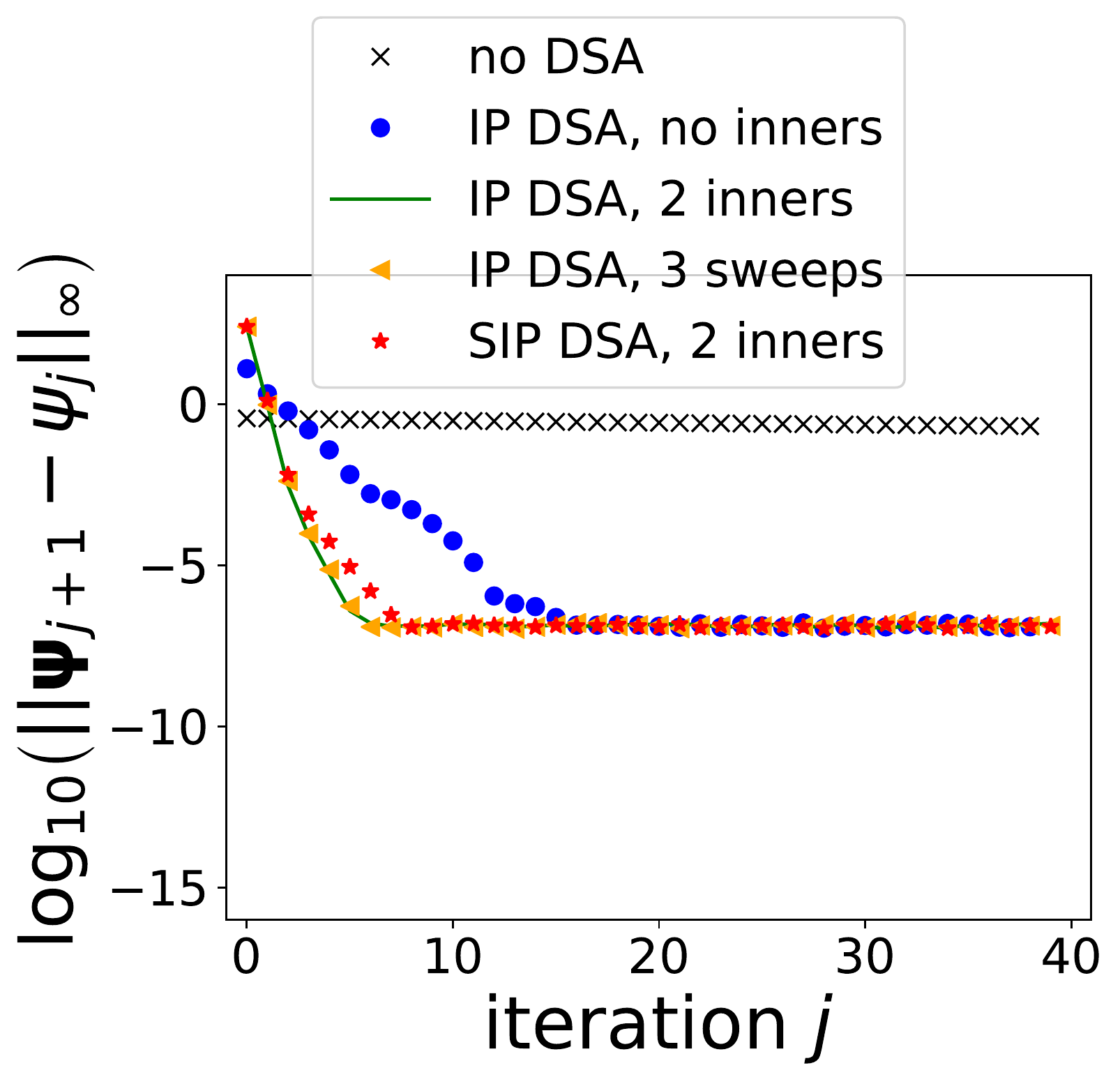} }
\subfloat[$\varepsilon = 10^{-3}$]
{ \includegraphics[width = 2in]{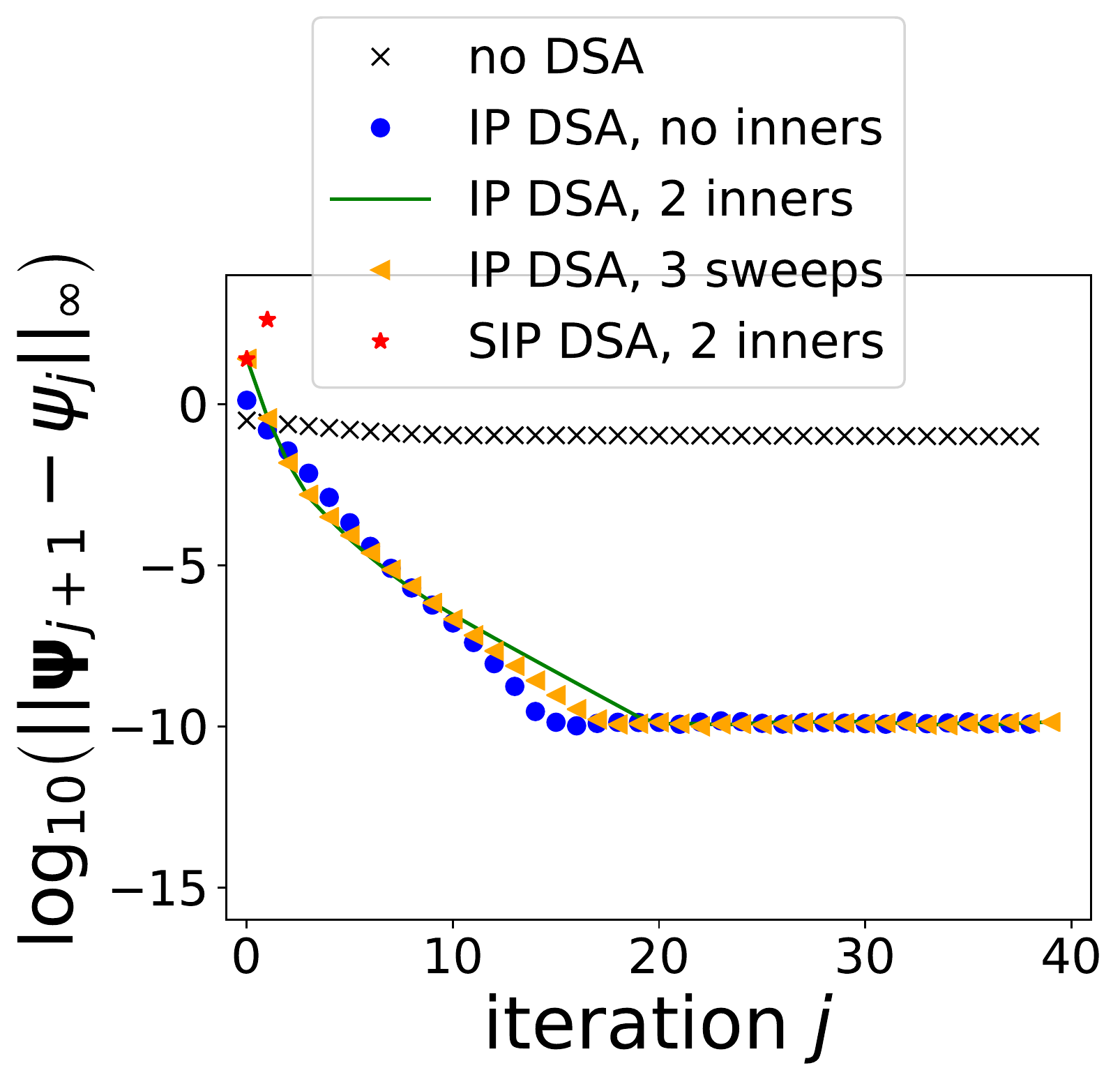}  }
\subfloat[$\varepsilon = 10^{-2}$]
{ \includegraphics[width = 2in]{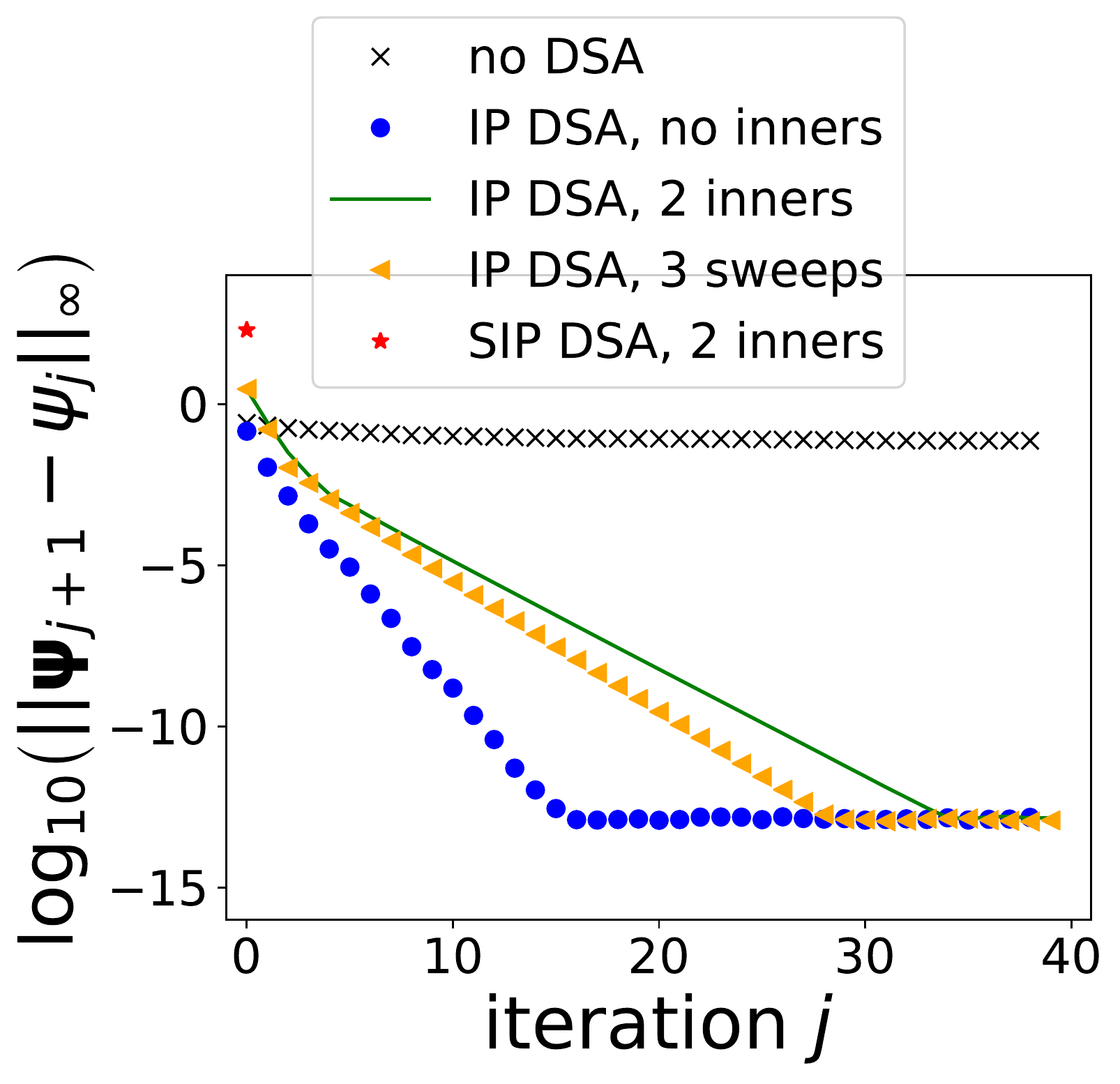} }\\
\subfloat[$\varepsilon = 10^{-1}$]
{ \includegraphics[width = 2in]{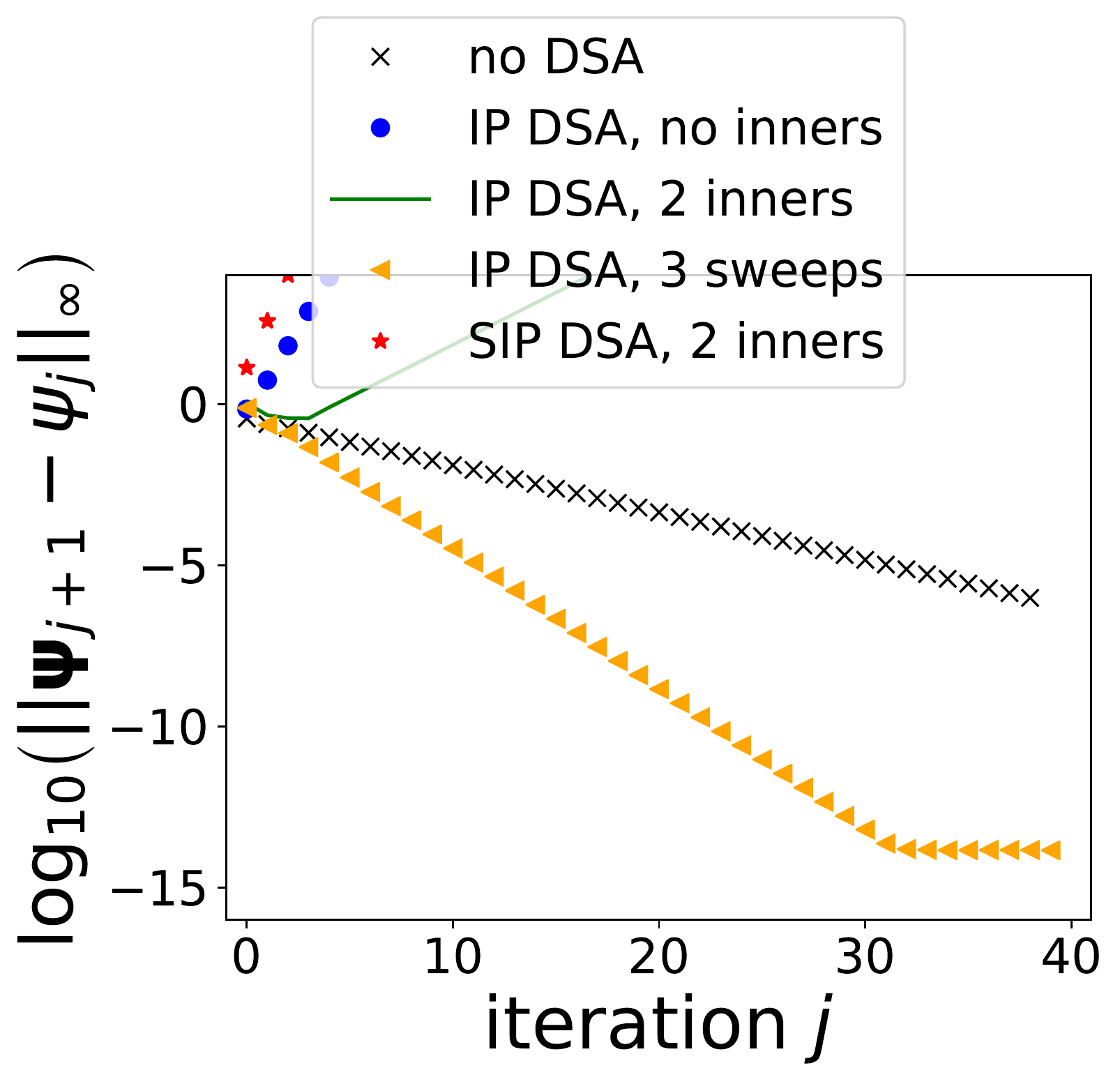} }
\subfloat[$\varepsilon = 0.75$]
{ \includegraphics[width = 2in]{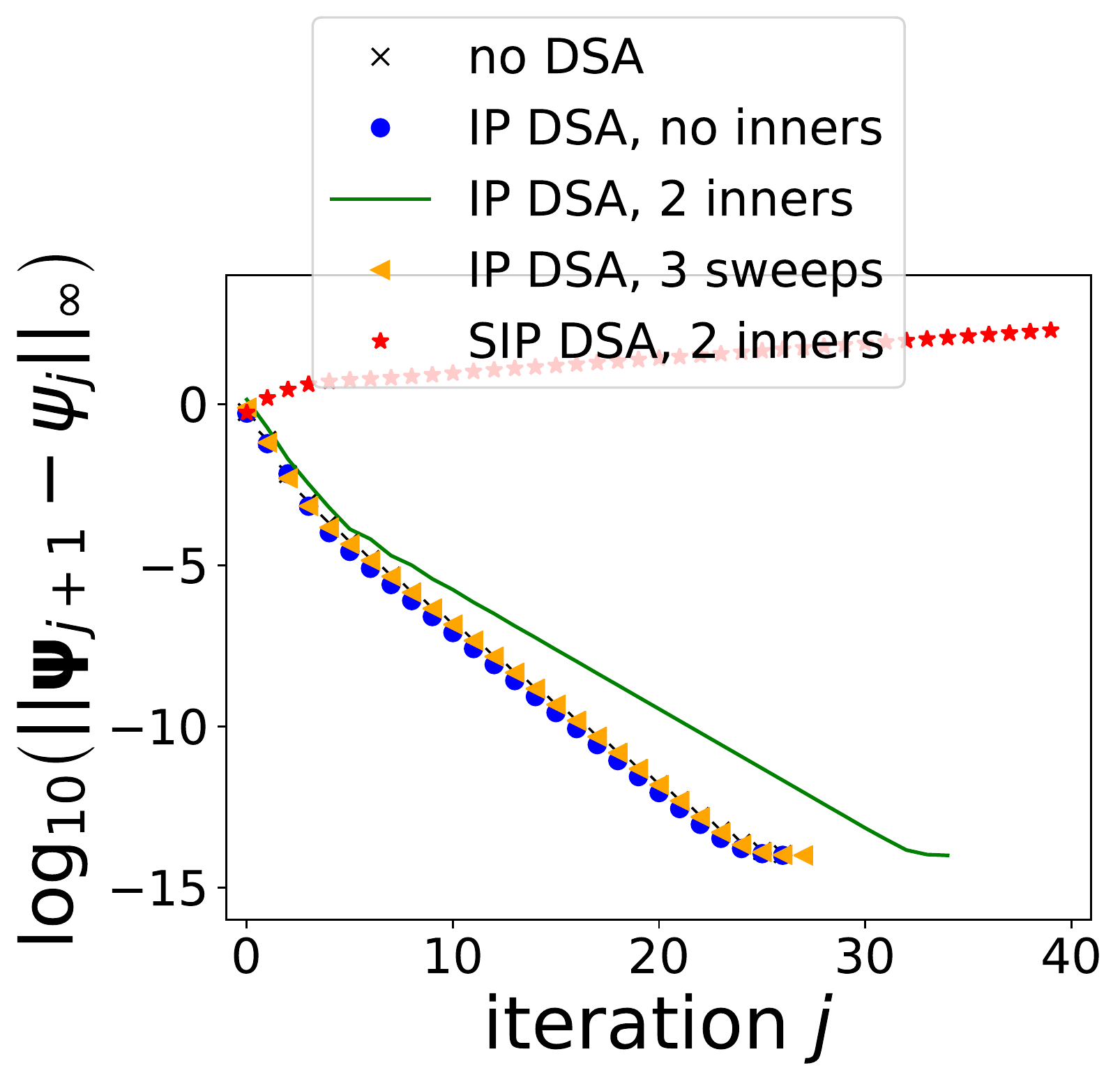} }
\caption{Iteration error estimate
         $\|\boldsymbol{\psi}_{j+1}-\boldsymbol{\psi}_{j+1}\|_{\infty}$ as a
         function of iteration index $j$ on the 3D Rayleigh-Taylor mesh,
         with and without DSA preconditioning using the DSA matrix
         \eqref{eq:modified DSA matrix}. }
\end{figure}

\begin{table}[!ht]
\label{tab_residual_3D}
\begin{center}
\caption{The final residual (\ref{eq:residual}) after 40 iterations with
         and without DSA preconditionin for the 3D Rayleigh-Taylor mesh.}
\begin{tabular}{|c|c|c|c|c|c|}
\hline 
$\varepsilon$ & IP DSA, 3 sweeps
              & IP DSA, 2 inners
              & SIP DSA, 2 inners
              & IP DSA, no inners
              & no DSA   \tabularnewline
\hline 
\hline 
0.75 & 9.97e-15 & {\color{black} 9.96e-15} & diverged & 9.90e-15 & 9.99e-15  \\
\hline 
1e-1 & 1.43e-14 & {\color{black} diverged} & diverged & diverged & 6.94e-07  \\
\hline 
1e-2 & 1.21e-13 & {\color{black} 1.45e-13} & diverged & 1.43e-13 & 7.31e-02  \\
\hline 
1e-3 & 1.36e-10 & {\color{black} 1.36e-10} & diverged & 1.14e-10 & 9.97e-02  \\
\hline 
1e-4 & 1.32e-07 & {\color{black} 1.57e-7} & 1.25e-07 & 1.43e-07 & 2.04e-01  \\
\hline 
\end{tabular}
\end{center}
\end{table}

\subsection{Additive DSA Preconditioning (Theorem \ref{th:new_DSA})}
\label{sec:numerical:th4}

\cred{In the thick regime, the DSA matrix derived in Theorem \ref{th:MIP} is effective when it can be readily inverted}; however, its condition number
scales like $1/\varepsilon$ (in addition to the standard $1/h^2$ scaling). In addition to the fact that even standard DG
discretizations of elliptic problems can be difficult for fast linear solvers such as AMG, inverting the discrete diffusion
operator derived here is not trivial. Theorem \ref{th:new_DSA} developed a two-part additive DSA preconditioner that
requires inverting a continuous  diffusion discretization three times and a mass-matrix like operator once, both of
which are more tractable to quickly (approximately) invert in parallel. This section demonstrates on a 1D-domain that
the derived two-part DSA preconditioner is indeed effective with respect to convergence of the larger iteration. \cred{Like the IP DSA 
variant, our results indicate that the new DSA variant based on Theorem \ref{th:new_DSA} is robust across all values of $\varepsilon$
(thin, intermediate, and thick). In contrast, the SIP DSA preconditioner actually results in a divergent iteration for $\varepsilon = 1e-3$
in the problem detailed below (however, as detailed in \cite{Wang-Ragusa-2010}, the SIP DSA preconditioner can be stabilized outside of
the thick limit by modifying the penalty coefficient appropriately}. {\color{black} We note that the variant of DSA based on Theorem \ref{th:new_DSA} is still significantly less effective as the IP DSA variant when $\varepsilon$ is in the intermediate regimes.  } 

Let $\sigma_a$ and $\sigma_t$ be defined as before, with zero inflow boundary conditions, and source term
\begin{align*}
q(x,mu) = \varepsilon \left(2 \sin(3x^2)^2 + \cos(x/3)^2\right).
\end{align*}
Figure \ref{fig:th4} plots the global residual as a function of iteration number for no DSA, SIP DSA, NIP DSA, and DSA based on Theorem \ref{th:new_DSA}.
\cred{ We use $6^{\text{th}}$ order local basis functions,
$100$ mesh elements, and an S4 quadrature discretization. } {\color{black} Note that, in \ref{fig:th4}, the caption ``IP DSA'' denotes the same option as the ``IP DSA, 2 inners'' option from the 2D and 3D results, since in the 1D case there is no need to handle mesh cycles in the transport sweep.}

\begin{figure}
\centering
\label{fig:th4}
\subfloat[$\varepsilon = 10^{-4}$]{ \includegraphics[width = 2in]{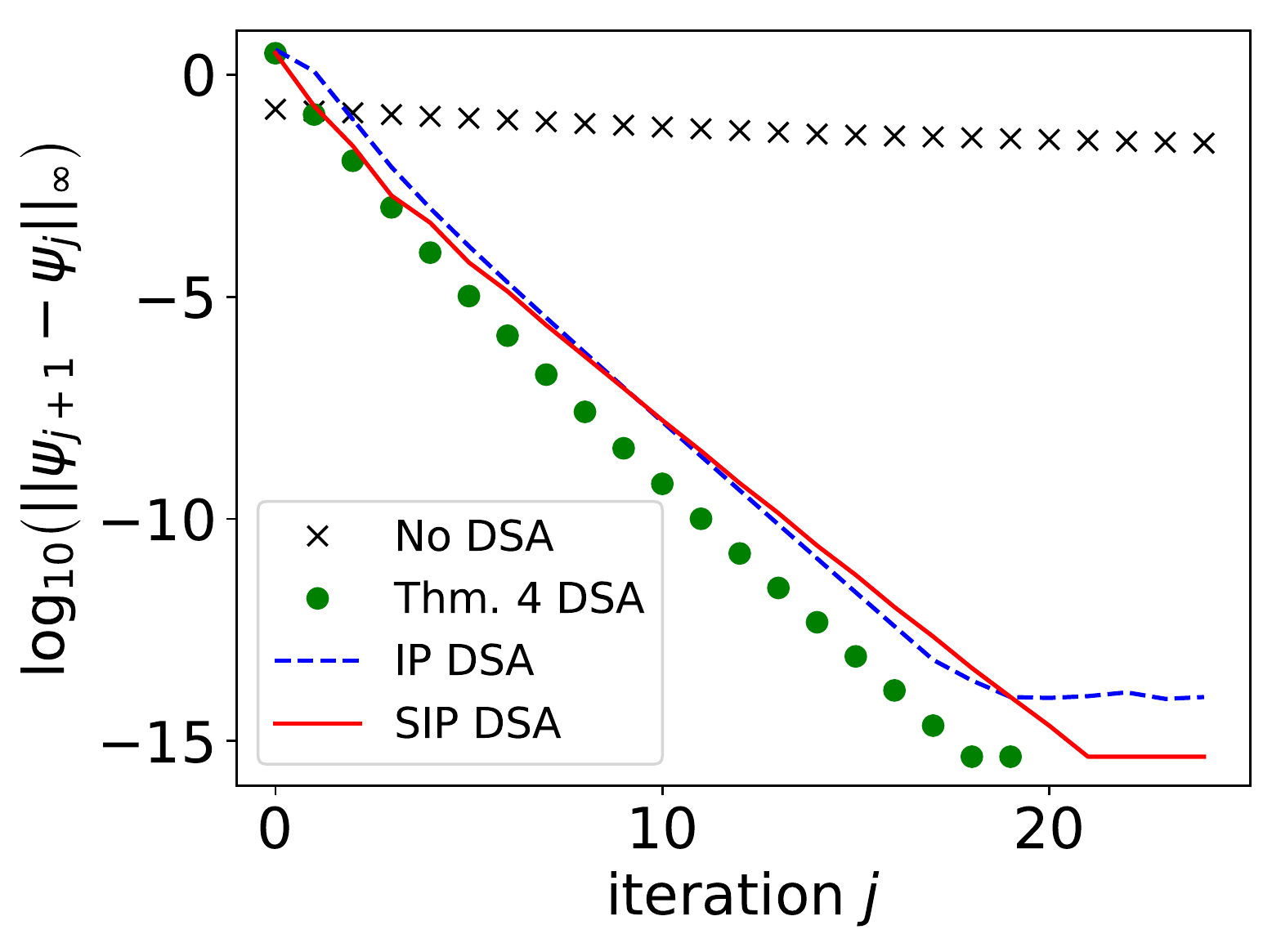} }
\subfloat[$\varepsilon = 10^{-3}$]{ \includegraphics[width = 2in]{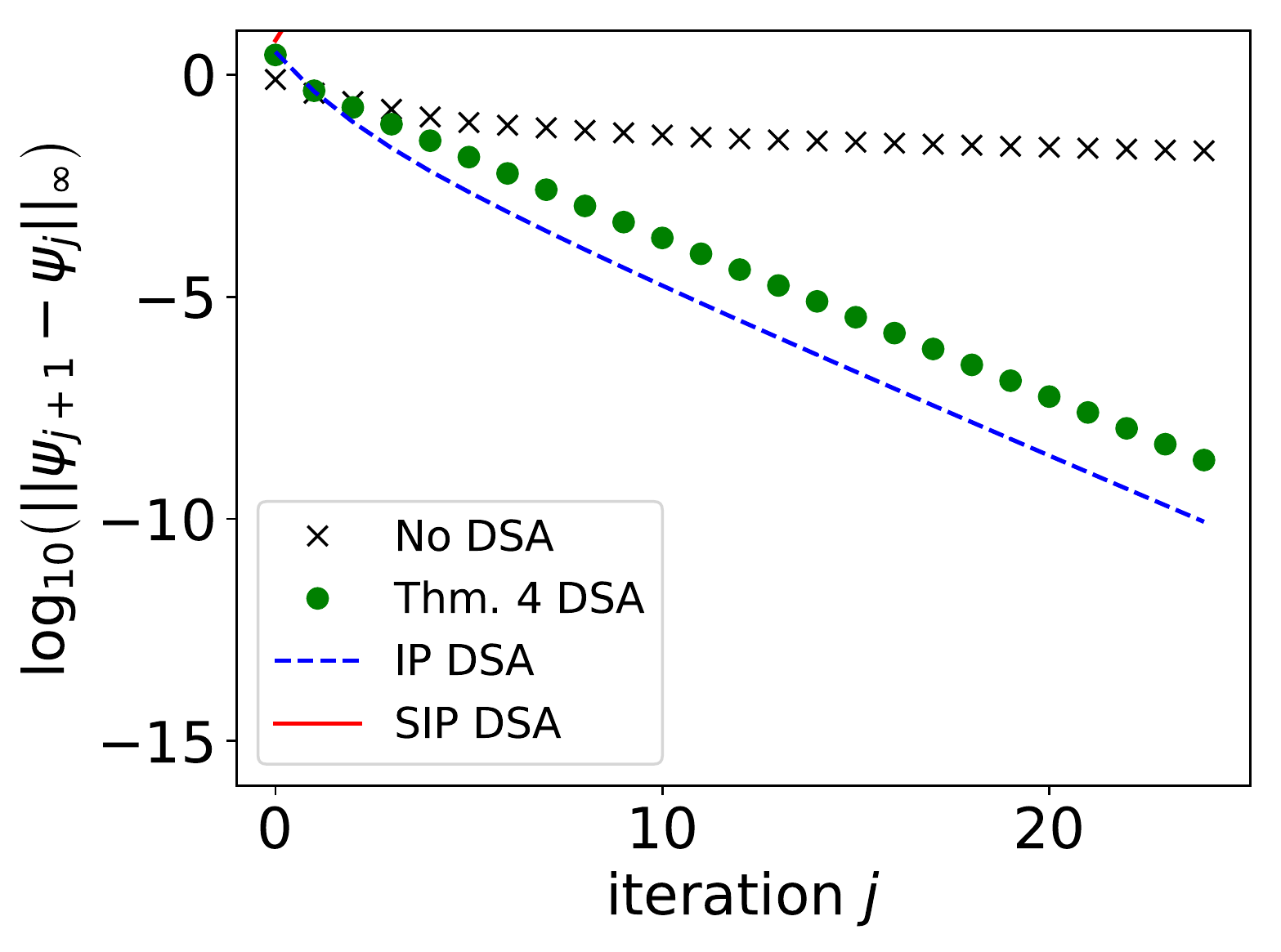}  }
\subfloat[$\varepsilon = 10^{-2}$]{ \includegraphics[width = 2in]{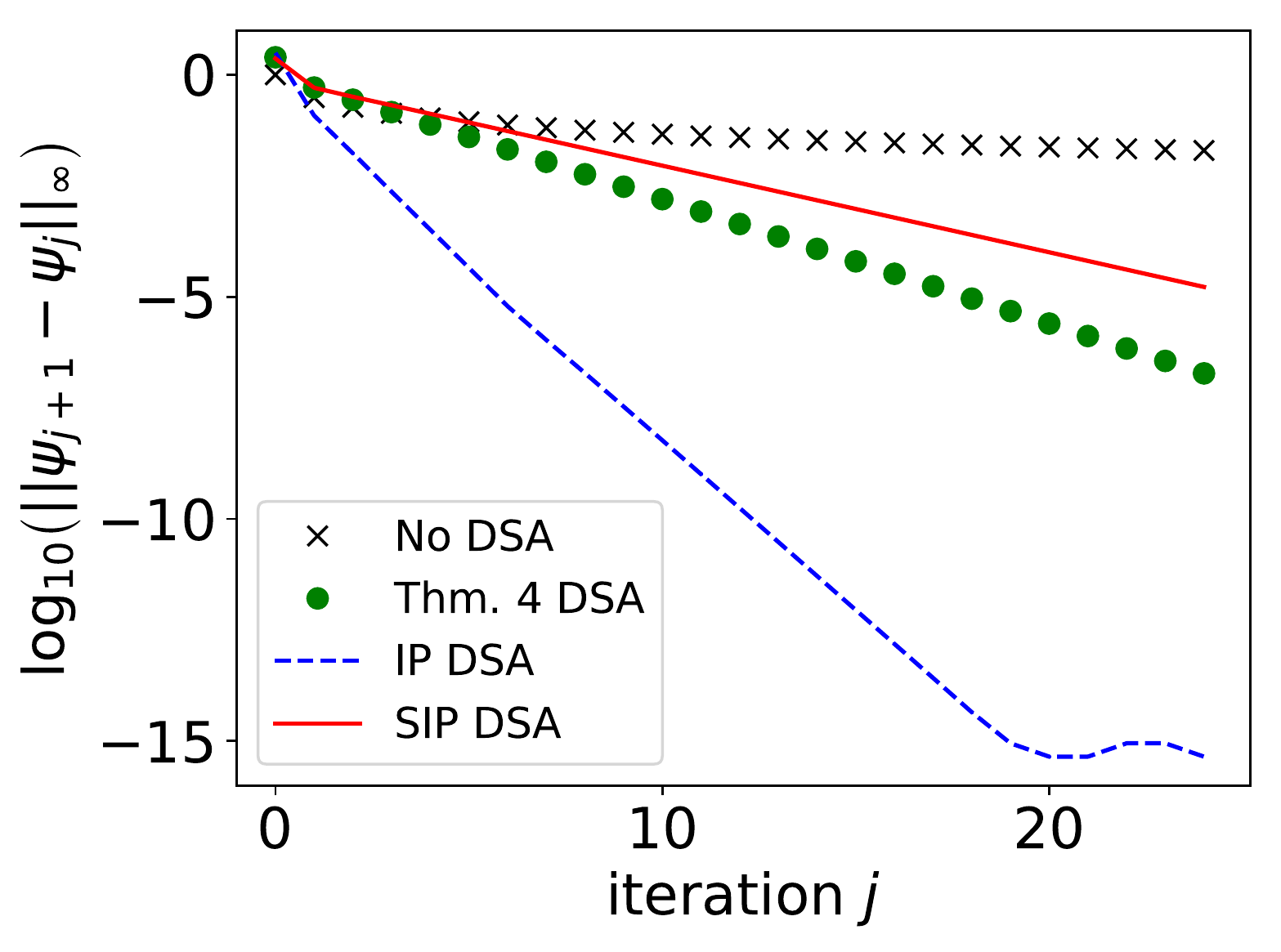} }\\
\subfloat[$\varepsilon = 0.1$]{ \includegraphics[width = 2in]{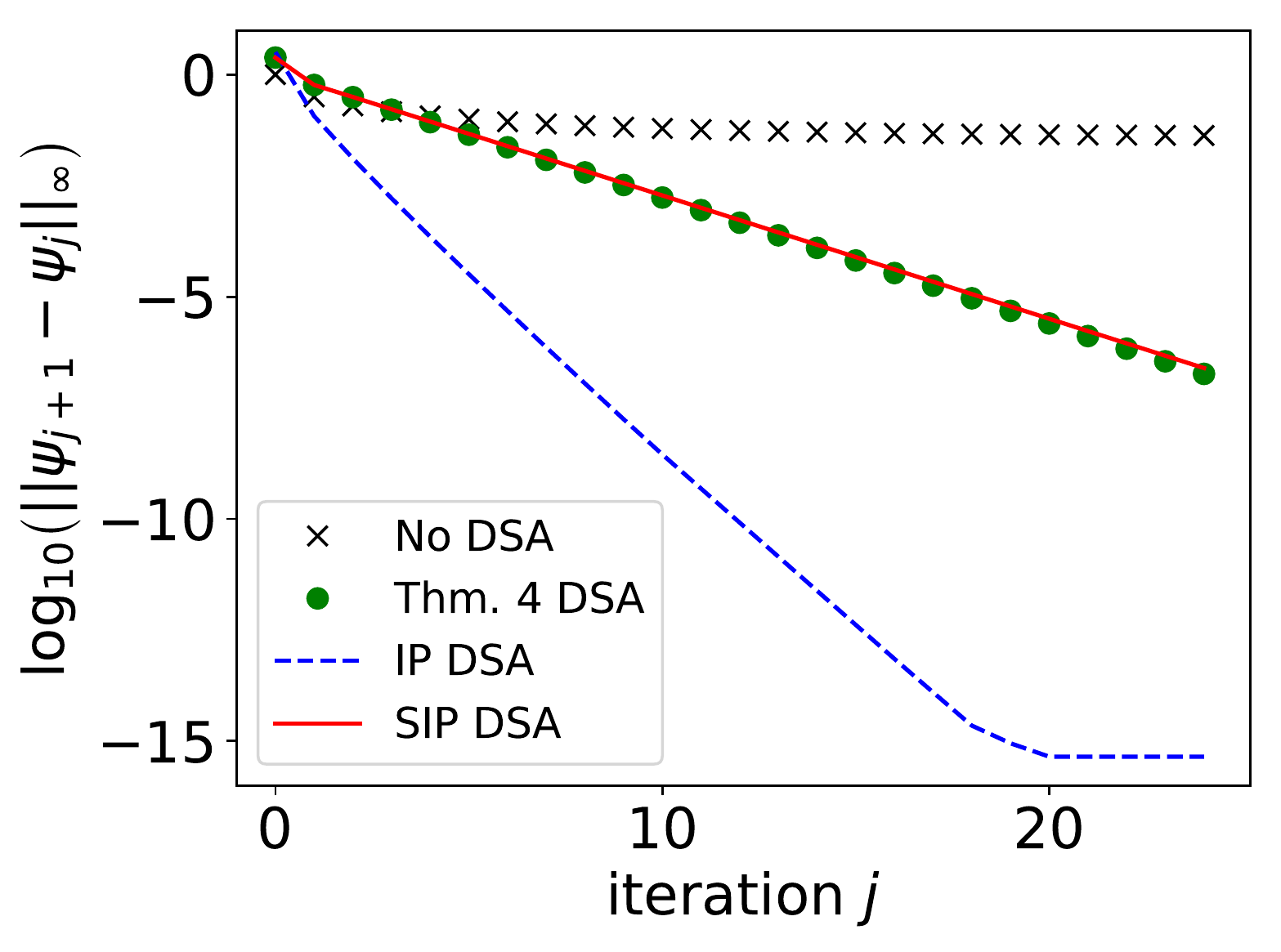} }
\subfloat[$\varepsilon = 0.5$]{ \includegraphics[width = 2in]{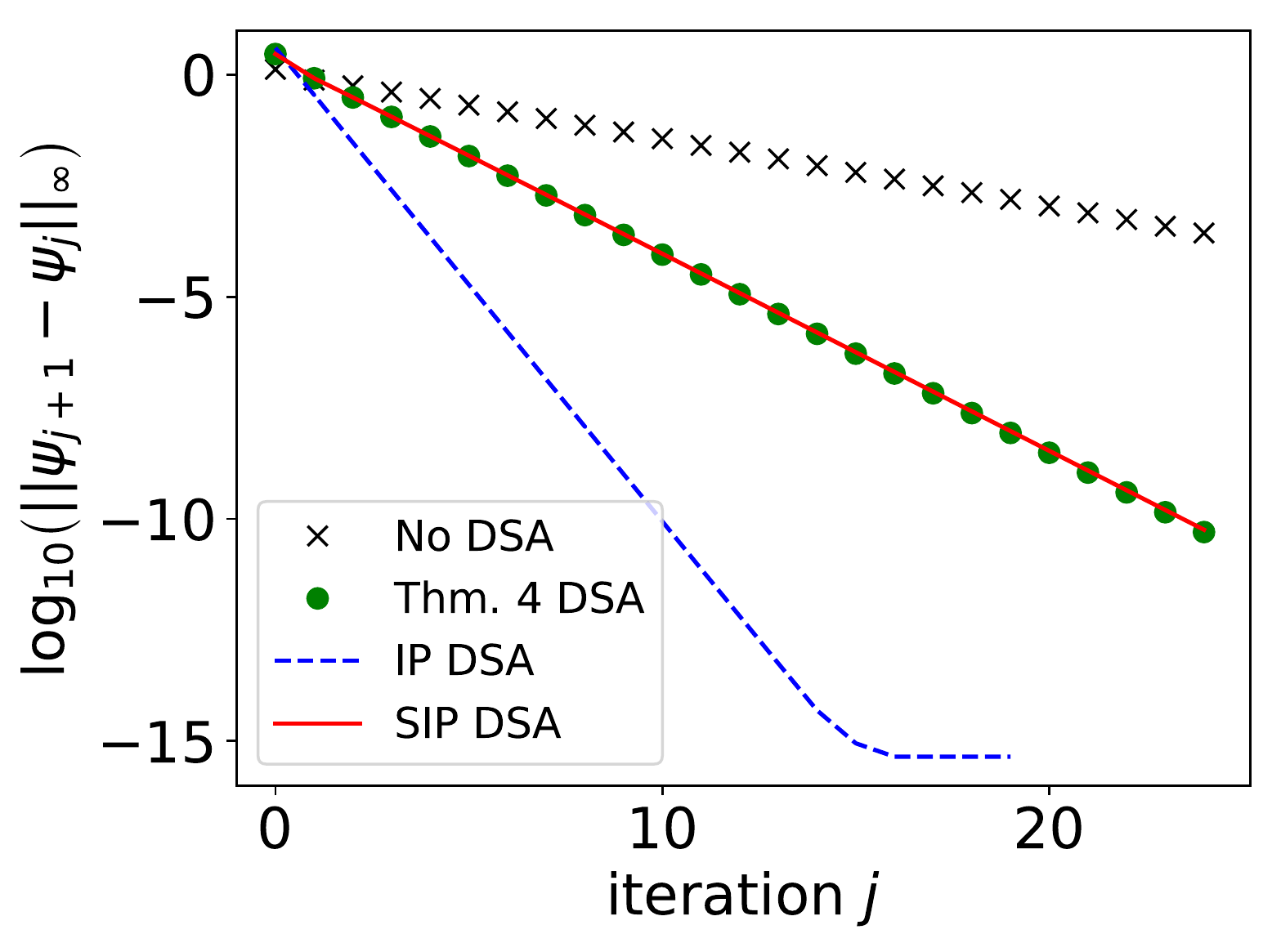} }
\subfloat[$\varepsilon = 0.75$]{ \includegraphics[width = 2in]{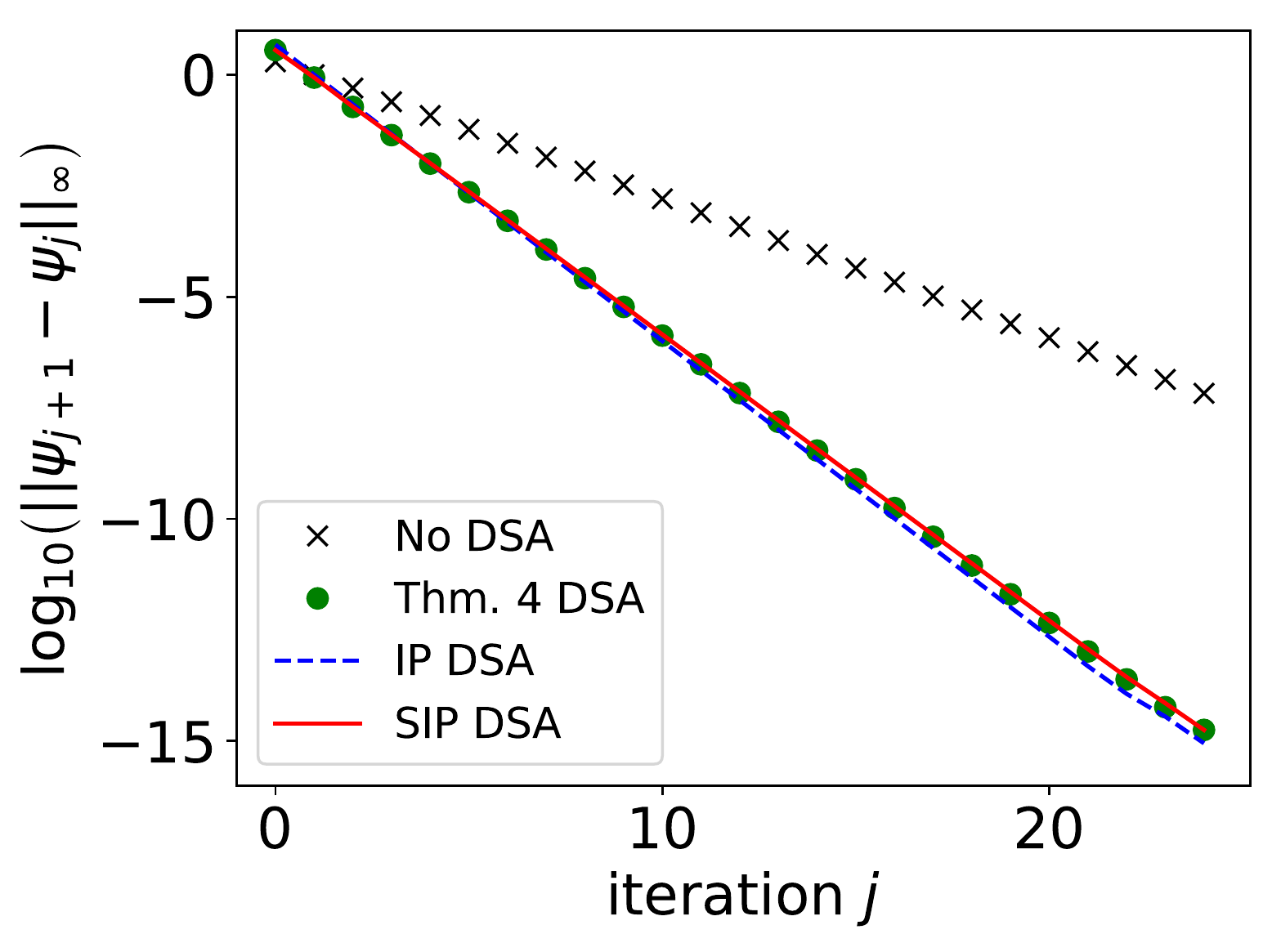} }
\caption{$\ell^\infty$-error, $\|\boldsymbol{\hat{\psi}} - \boldsymbol{\psi}_{j}\|_{\infty}$, as a function of iteration
number $j$, where $\boldsymbol{\hat{\psi}}$ is the exact solution and $\boldsymbol{\psi}_{j}$ the
$j$th iterate. Error is shown for no DSA, SIP DSA, IP DSA, and a two-part DSA preconditioner based on
Theorem \ref{th:new_DSA}.}
\end{figure}

\section{Conclusions\label{sec:Conclusions}}

This paper derives a discrete analysis of DSA applied to high-order DG discretizations of the $S_N$ transport equations. 
The basis for DSA is taking a simple fixed-point ``source iteration,'' which is slow to converge, and recognizing that the
slowly decaying error modes can be represented by a certain diffusion operator. DSA then preconditions source iteration
with an appropriate diffusion solve, as a correction for these slowly decaying error modes. When the mean free path of
particles is very small, $\varepsilon \ll 1$, conditioning of source iteration is $\mathcal{O}(1/\varepsilon^2)$, and
DSA is critical for convergence.

Here, we derive a discrete representation of the slowly decaying error modes for small $\varepsilon$. This
leads to the development of a DSA preconditioner that resembles a symmetric interior penalty DG discretization
of diffusion-reaction, where the resulting (preconditioned) fixed-point iteration is conditioned like $1+\mathcal{O}(\varepsilon)$
(Theorem \ref{th:MIP}). However, applying this preconditioner requires inverting a DG matrix that is ill-conditioned,
$\kappa \sim\mathcal{O}(1/\varepsilon)$, and, furthermore, elliptic DG discretizations are often difficult for fast preconditioners
such as multigrid. This motivates further analysis, where a two-part additive DSA preconditioner is developed based on
solving a continuous Galerkin (CG) discretization of diffusion-reaction, in addition to a second term that involves
two CG solves, and one solve of a mass-matrix-like term. These solves are now all conditioned independent of
$\varepsilon$ and more amenable to fast solvers such as multigrid. Furthermore, the preconditioner leads to a
larger fixed-point iteration that is well conditioned, $\kappa \sim 1+\mathcal{O}(\varepsilon)$, and will converge
rapidly for small $\varepsilon$ (Theorem \ref{th:new_DSA}). 

Finally, there is larger interest in discretizing HO DG on HO (curved) meshes. Source iteration relies on the
discretization of advection being block triangular in some ordering and, therefore, easily invertible. However,
HO meshes lead to cycles in the mesh, and the resulting discretization of advection in the transport equations
is no longer block triangular. When cycles are present, a method to approximate the inversion of advection
in source iteration through a pseudo-optimal Gauss-Seidel has been developed in \cite{sweep18}. Theorem
\ref{th:mesh_cycles} extends the handling of cycles to cases where DSA is necessary, proving that cycles can
be accounted for by performing an additional two source iterations for each larger DSA iteration. 

\section{Appendix\label{sec:Appendix}}

First we introduce a technical Lemma regarding the linear system
$(I-T_{\varepsilon})\boldsymbol{\psi}^{(d)} = \tilde{\boldsymbol{q}}^{(d)}$,
see (\ref{eq:psi_d}). This then leads to Proposition
\ref{prop:condition number}, which proves that the conditioning of
$(I-T_{\varepsilon})$ is $\mathcal{O}(\varepsilon^{-2}),$ making
effective preconditioning critical for small $\varepsilon$. 
\begin{lemma}
\label{lem:Neumann expansion}Define 
\begin{align}
H^{(d)} & =M_{t}^{-1}\left(\boldsymbol{\Omega}_{d}\cdot\mathbf{G}+F^{(d)}\right),\label{eq:Hd}\\
c_{0} & =\max\left\{ \max_{d}\|H^{(d)}\|,\|M_{t}^{-1}M_{a}\|\right\} ,\nonumber 
\end{align}
and assume that $\varepsilon\|H^{(d)}\|<1.$ Then, the operator in
(\ref{eq:psi_d}) satisfies
\begin{align}
\left( (I-T_{\varepsilon})\boldsymbol{\psi} \right)^{(d)} & =
\left(\boldsymbol{\psi}^{(d)}-\frac{1}{4\pi}\boldsymbol{\varphi}\right)+\varepsilon\frac{1}{4\pi}H^{(d)}\boldsymbol{\varphi}-\nonumber \\
 & \,\,\,\,\,\,\,\,\frac{1}{4\pi}\varepsilon^{2}\left(\left(H^{(d)}\right)^{2}-M_{t}^{-1}M_{a}\right)\boldsymbol{\varphi}+R_{\varepsilon}^{(d)},\label{eq:Neumann expansion for T}
\end{align}
where the norm of the remainder $R_{\varepsilon}^{(d)}$ is bounded
by
\[
\left\Vert R_{\varepsilon}^{(d)} \right\Vert \leq\varepsilon^{3}\frac{1}{4\pi}\left(\frac{c_{0}^{3}}{1-\varepsilon c_{0}}\left(1+\varepsilon^{2}c_{0}\right)+\left(c_{0}^{2}+\varepsilon c_{0}^{3}\right)\right) \| \boldsymbol{\varphi} \|  .
\]
\end{lemma}
\begin{proof}
Note the matrix identity,
\begin{equation}
\left(I+\varepsilon H^{(d)}\right)^{-1}=I-\varepsilon H^{(d)}+\varepsilon^{2}\left(H^{(d)}\right)^{2}-\varepsilon^{3}\left(H^{(d)}\right)^{3}\left(I+\varepsilon H^{(d)}\right)^{-1}.\label{eq:IplusH_inv}
\end{equation}
Plugging into the definition of $T_{\varepsilon}$ and expanding yields
\begin{align*}
\left( (I-T_{\varepsilon})\boldsymbol{\psi}\right)^{(d)} & =\left[I-\frac{1}{4\pi}\left(I+\varepsilon H^{(d)}\right)^{-1}\left(I-\varepsilon^{2}M_{t}^{-1}M_{a}\right)P_{0}\right]\boldsymbol{\psi}^{(d)}\\
 & =\boldsymbol{\psi}^{(d)}-\frac{1}{4\pi}\left[I-\varepsilon H^{(d)}+\varepsilon^{2}\left(\left(H^{(d)}\right)^{2}-M_{t}^{-1}M_{a}\right)-\right.\\
 & \qquad\varepsilon^{3}\left(\left(H^{(d)}\right)^{3}\left(I+\varepsilon H^{(d)}\right)^{-1}-H^{(d)}M_{t}^{-1}M_{a}\right)-\\
 & \qquad\left.\varepsilon^{4}\left(H^{(d)}\right)^{2}M_{t}^{-1}M_{a}+\varepsilon^{5}\left(H^{(d)}\right)^{3}\left(I+\varepsilon H^{(d)}\right)^{-1}M_{t}^{-1}M_{a}\right]\boldsymbol{\varphi}.
\end{align*}
Equation (\ref{eq:Neumann expansion for T}) consists of terms up
to $\mathcal{O}(\varepsilon^{2})$. Collecting higher-order terms
yields the remainder term, $R_{\varepsilon}^{(d)}$, given by
\begin{align*}
R_{\varepsilon}^{(d)} & =\frac{1}{4\pi}\varepsilon^{3}\left(H^{(d)}\right)^{3}\left(I+\varepsilon H^{(d)}\right)^{-1}\left(I-\varepsilon^{2}M_{t}^{-1}M_{a}\right) \boldsymbol{\varphi}+\\
 & \,\,\,\,\,\,\varepsilon^{3}H^{(d)}\left(I-\varepsilon H^{(d)}\right)M_{t}^{-1}M_{a} \boldsymbol{\varphi}.
\end{align*}
The bound on $\left\Vert R_{\varepsilon}^{(d)}\right\Vert $ follows
from the identity $\left\Vert \left(I+\varepsilon H^{(d)}\right)^{-1}\right\Vert \leq\frac{1}{1-\varepsilon\left\Vert H^{(d)}\right\Vert }.$\end{proof}

Before stating Proposition~\ref{prop:condition number}, we set up
preliminary notation. First, the linear system (\ref{eq:psi_d})
can be written in the form
\begin{equation}
I-T_{\varepsilon}=I-H_{\varepsilon}P_{0}, \label{eq:I-H*P 2}
\end{equation}
where $T_{\varepsilon}=H_{\varepsilon}P_{0}$ and $H_{\varepsilon}$
is defined via
\[
  \left( H_{\varepsilon}\boldsymbol{\psi} \right)^{(d)} = 
  \left(I+\varepsilon M_{t}^{-1}\left(\boldsymbol{\Omega}_{d} \cdot \mathbf{G}+F^{(d)}\right)\right)^{-1}\frac{1}{4\pi}\left(I-\varepsilon^{2}M_{t}^{-1}M_{a}\right)\boldsymbol{\psi}^{(d)},
\]
for $d=1,\ldots,N_{\Omega}$. 

Notice that
\begin{equation}
\left(P_{0}\left(I-T_{\varepsilon}\right)P_{0}\boldsymbol{\psi}\right)^{(d)}=\left(I-S_{\varepsilon}\right)\left( P_{0}\boldsymbol{\psi} \right)^{(d)},\label{eq:Te_Se}
\end{equation}
where $S_{\varepsilon}$ is defined in equation (\ref{eq:sweep matrix-2}).
Also, from Lemma~\ref{lem:Neumann expansion},
\begin{equation}
I-T_{\varepsilon}=Q_{0}+\varepsilon H_{0}P_{0}+\varepsilon^{2}H_{1}P_{0}+\mathcal{O}\left(\varepsilon^{3}\right),\label{eq:Neumann expansion, proof}
\end{equation}
where $H_{i}$ denotes block-diagonal in $d$ matrices, for $i=1,...,N_{\Omega}$
as in (\ref{eq:Neumann expansion for T}); in particular,
\[
\left( H_{0} \right)_{d,d} =\frac{1}{4\pi}H^{(d)}=\frac{1}{4\pi}M_{t}^{-1}\left(\boldsymbol{\Omega}_{d}\cdot\mathbf{G}+F^{(d)}\right), \,\,\,\,\, \left( H_{1} \right)_{d,d} = \frac{1}{4\pi} \left(\left(H^{(d)}\right)^{2}-M_{t}^{-1}M_{a}\right) .
\]
Finally, define $F_{0}=\frac{1}{4\pi}\sum_{d}w_{d}F^{(d)}$. Then
using $\sum_{d}w_{d}\boldsymbol{\Omega}_{d}=\mathbf{0}$, it follows
that
\begin{equation}
\left(P_{0}H_{0}P_{0}\boldsymbol{\psi}\right)^{(d)}=M_{t}^{-1}F_{0}\boldsymbol{\varphi},\,\,\,\,d=1,\ldots,N_{\Omega}.\label{eq:P0H0P0}
\end{equation}

We now prove Proposition~\ref{prop:condition number}.

\begin{proof}
First, choose some unit norm vector $\boldsymbol{\psi}$ for which
$Q_{0}\boldsymbol{\psi}=\boldsymbol{\psi}$. Then, using equation
(\ref{eq:I-H*P 2}),
\[
\left\Vert I-T_{\varepsilon}\right\Vert_W \geq\left\Vert \left(I-T_{\varepsilon}\right)Q_{0}\boldsymbol{\psi}\right\Vert_W =\left\Vert Q_{0}\boldsymbol{\psi}\right\Vert _W =1.
\]
For the inverse, $(I-T_{\varepsilon})\mathbf{x}=\mathbf{y}$ can be
decomposed based on $P_{0}$ and $Q_{0}$ via $(I-T_{\varepsilon})(P_{0}\mathbf{x}+Q_{0}\mathbf{x})=(P_{0}\mathbf{y}+Q_{0}\mathbf{y})$.
Multiplying on the left by the full-column-rank operator $(P_{0};Q_{0})$
yields the equivalent linear system
\begin{equation}
\begin{pmatrix}P_{0}\\
Q_{0}
\end{pmatrix}(I-T_{\varepsilon})\begin{pmatrix}P_{0} & Q_{0}\end{pmatrix}\begin{pmatrix}P_{0}\mathbf{x}\\
Q_{0}\mathbf{x}
\end{pmatrix}=\begin{pmatrix}P_{0}\\
Q_{0}
\end{pmatrix}(P_{0}\mathbf{y}+Q_{0}\mathbf{y}).\label{eq:blocksystem}
\end{equation}
Denote $\mathbf{x}_{P}=P_{0}\mathbf{x}$ and $\mathbf{x}_{Q}=Q_{0}\mathbf{x}$,
and likewise for $\mathbf{y}$. Then (\ref{eq:blocksystem}) yields
a $2\times2$ set of equations, which, noting the expansion from Lemma
\ref{lem:Neumann expansion} and (\ref{eq:Neumann expansion, proof})
and the orthogonality of $P_{0}$ and $Q_{0}$, reduces to
\begin{equation}
\begin{pmatrix}P_{0}(I-T_{\varepsilon})P_{0} & \mathbf{0}\\
Q_{0}(I-T_{\varepsilon})P_{0} & Q_{0}
\end{pmatrix}\begin{pmatrix}\mathbf{x}_{p}\\
\mathbf{x}_{Q}
\end{pmatrix}=\begin{pmatrix}\mathbf{y}_{P}\\
\mathbf{y}_{Q}
\end{pmatrix}.\label{eq:2x2}
\end{equation}
Here, $\mathbf{x}_{P}$ is fully determined by inverting $P_{0}(I-T_{\varepsilon})P_{0}$
on the range of $P_{0}$. This is equivalent to inverting $I-S_{\varepsilon}$
(\ref{eq:Te_Se}), which is assumed to be full rank. Now, choose some
vector $\hat{\mathbf{x}}$ for which $P_{0}\boldsymbol{\hat{\mathbf{x}}}=\boldsymbol{\hat{\mathbf{x}}}$,
where each direction block $\hat{\mathbf{x}}_{d}$ corresponds to
a continuous function. From (\ref{eq:F*P and PT * tildeFd}), we have
that $F_{0}P_{0}\hat{\mathbf{x}}=\mathbf{0}$ and from (\ref{eq:P0H0P0})
$P_{0}H^{(d)}P_{0}\hat{\mathbf{x}}=\mathbf{0}$. From equation (\ref{eq:Neumann expansion, proof}),
this yields
\begin{equation}
\left(P_{0}\left(I-T_{\varepsilon}\right)P_{0}\hat{\mathbf{x}}\right)=\mathcal{O}\left(\varepsilon^{2}\right)P_{0}\hat{\mathbf{y}}.\label{eq:haty}
\end{equation}
Recall by orthogonality, $\|\mathbf{y}\|_{W}=\|\mathbf{y}_{P}\|_{W}+\|\mathbf{y}_{Q}\|_{W}$.
Now define a vector $\tilde{\mathbf{y}}$ such that $\tilde{\mathbf{y}}_{P}=P_{0}\hat{\mathbf{y}}$
from (\ref{eq:haty}) and $\tilde{\mathbf{y}}_{Q}=\mathbf{0}$, and
let $\tilde{\mathbf{x}}=(I-T_{\varepsilon})^{-1}\tilde{\mathbf{y}}$.
Then, in the notation of (\ref{eq:2x2}),
\begin{align*}
\|(I-T_{\varepsilon})^{-1}\|_{W} & =\sup_{\mathbf{\|y\|_{W}}=1}\|(I-T_{\varepsilon})^{-1}\mathbf{y}\|_{W}=\sup_{\mathbf{\|y\|_{W}}=1}\|\mathbf{x}_{P}\|_{W}+\|\mathbf{x}_{Q}\|_{W}\\
 & \qquad\geq\|\tilde{\mathbf{x}}_{P}\|_{W}+\|\tilde{\mathbf{x}}_{Q}\|_{W}=\mathcal{O}(\varepsilon^{-2})\|\hat{\mathbf{x}}_{P}\|_{W}+\|\tilde{\mathbf{x}}_{Q}\|_{W}\\
 & \qquad\geq\mathcal{O}(\varepsilon^{-2})\|\hat{\mathbf{x}}_{P}\|_{W}=\mathcal{O}(\varepsilon^{-2}).
\end{align*}

To prove equation (\ref{eq:DSA system}), note that
\begin{align*}
\tilde{A}_{\varepsilon} & =\left(E_{\varepsilon}P_{0}+Q_{0}\right)\left(I-T_{\varepsilon}\right)\\
 & =E_{\varepsilon}P_{0}\left(I-T_{\varepsilon}\right)P_{0}+Q_{0}\left(I-T_{\varepsilon}\right)\\
 & =P_{0}+\mathcal{O}\left(\varepsilon\right)+Q_{0}-Q_{0}T_{\varepsilon}\\
 & =I+\mathcal{O}\left(\varepsilon\right)-Q_{0}T_{\varepsilon}\\
 & =I+\mathcal{O}\left(\varepsilon\right).
\end{align*}
In the second equality, we used the assumption that
$E_{\varepsilon}P_{0}\left(I-T_{\varepsilon}\right)P_{0}=P_{0}+\mathcal{O}\left(\varepsilon\right)$ and
the identity $P_{0}\left(I-T_{\varepsilon}\right)=P_{0}\left(I-T_{\varepsilon}\right)P_{0}.$
In the last equality, we used that $Q_{0}T_{\varepsilon}=\mathcal{O}\left(\varepsilon\right),$which
follows from equation (\ref{eq:Neumann expansion, proof}).
\end{proof}
\begin{rem}
Letting $h_{\mathbf{x}}$ denote the characteristic mesh spacing,
the assumption $\varepsilon\left\Vert H^{(d)}\right\Vert <1$ in Lemma~\ref{lem:Neumann expansion}
holds if $\varepsilon\lesssim\sigma_{t}h_{\mathbf{x}}$, which corresponds
to the optically thick limit.
\end{rem}

\section*{Acknowledgements}
We would like to thank Jim Warsa for pointing out the equivalence between the nonsymmetric interior penalty method explored in this paper 
and his consistent P1 diffusion discretization, as well as helping us understand many of the nuances of $S_N$ transport preconditioning. We would also like to thank Jim Morel for his many insightful comments,
and for suggesting the need to perform additional transport sweeps when there are mesh cycles.

\bibliographystyle{plain}
\bibliography{pgm_references}

\end{document}